\documentclass[a4paper]{amsart}

\usepackage{hyperref}
\usepackage{url}
\usepackage{amssymb,latexsym}
\usepackage[utf8x]{inputenc}
\usepackage{amsmath,amsthm}
\usepackage{amsfonts,mathrsfs}
\usepackage[capitalize]{cleveref}
\usepackage{enumerate,units}
\usepackage[all]{xy}
\usepackage{graphicx}
\usepackage{xcolor,url}
\usepackage{tikz-cd}
\usepackage[title]{appendix}
\usepackage[british]{babel}
\usepackage{etoolbox}

\newtoggle{THESIS}

\theoremstyle{plain}
\newtheorem{theorem}{Theorem}[section]
\newtheorem{corollary}[theorem]{Corollary}
\newtheorem{lemma}[theorem]{Lemma}
\newtheorem{proposition}[theorem]{Proposition}
\newtheorem{remarkcnt}[theorem]{Remark}
\newtheorem{notation}[theorem]{Notation}
\newtheorem{fact}[theorem]{Fact}

\newtheorem*{theorem*}{Theorem}
\newtheorem{question}[theorem]{Question}

\theoremstyle{definition}
\newtheorem{definition}[theorem]{Definition}

\theoremstyle{remark}

\newcommand{\Q}{\mathbb{Q}}
\newcommand{\R}{\mathbb{R}}
\newcommand{\F}{\mathbb{F}}

\newcommand{\set}[1]{\left\{ {#1} \right\}}
\newcommand{\vect}[1]{\langle {#1} \rangle}
\newcommand{\abs}[1]{\lvert {#1} \rvert}

\newcommand{\ACF}{\mathrm{ACF}}
\newcommand{\PAC}{\mathrm{PAC}}
\newcommand{\PSF}{\mathrm{PSF}}
\newcommand{\ACFA}{\mathrm{ACFA}}
\newcommand{\NSOP}{\mathrm{NSOP}}
\newcommand{\characteristic}{\mathrm{char}}
\newcommand{\trdeg}{\mathrm{trdeg}}
\newcommand{\tp}{\mathrm{tp}}
\newcommand{\acl}{\mathrm{acl}}
\newcommand{\dcl}{\mathrm{dcl}}
\newcommand{\M}{\mathbb{M}}

\newcommand{\fr}{\mathrm{Frac}}
\newcommand{\Th}{\mathrm{Th}}

\newcommand{\forkindep}[1][]{%
  \mathrel{
    \mathop{
      \vcenter{
        \hbox{\oalign{\noalign{\kern-.3ex}\hfil$\vert$\hfil\cr
              \noalign{\kern-.7ex}
              $\smile$\cr\noalign{\kern-.3ex}}}
      }
    }\displaylimits_{#1}
  }
}

\title{On algebraically closed fields with a distinguished subfield}

\author{Christian d\textquoteright Elb\'ee}
\address{Fields Institute for Research in Mathematical Sciences Office 416
222 College Street. Toronto, Ontario. Canada.}
\email{cdelbee@fields.utoronto.ca}
\urladdr{\href{http://choum.net/\textasciitilde chris/page\textunderscore perso/}{http://choum.net/\textasciitilde chris/page\textunderscore perso/}}

\author{Itay Kaplan}
\address{Einstein Institute of Mathematics, Hebrew University of
Jerusalem, 91904, Jerusalem Israel.}
\email{kaplan@math.huji.ac.il}
\urladdr{\href{http://math.huji.ac.il/~kaplan}{math.huji.ac.il/~kaplan}}

\author{Leor Neuhauser}
\address{Einstein Institute of Mathematics, Hebrew University of
Jerusalem, 91904, Jerusalem Israel.}
\email{leor.neuhauser@math.huji.ac.il}

\thanks{
The authors would like to thank the Israel Science Foundation for their
support of this research (grant no. 1254/18). The first-named author was partially supported by the S.A Schonbrunn Fellowship.
This paper was done as part of the third-named author's master thesis under the supervision of the first- and second-named authors.}


\begin{document}

\begin{abstract}
    This paper is concerned with the model-theoretic study of pairs $(K,F)$ where $K$ is an algebraically closed field and $F$ is a distinguished subfield of $K$ allowing extra structure. We study the basic model-theoretic properties of those pairs, such as quantifier elimination, model-completeness and saturated models. We also prove some preservation results of classification-theoretic notions such as stability, simplicity, NSOP$_1$, and NIP. As an application, we conclude that a PAC field is NSOP$_1$ iff its absolute Galois group is (as a profinite group). 
\end{abstract}

\maketitle


\setcounter{tocdepth}{1}
\tableofcontents

\section{Introduction}
\iftoggle{THESIS}{
In their study of pseudo-algebraically closed fields, or PAC fields (known at that time as regularly closed fields, for obvious reasons, see \cref{def PSF}) Cherlin, van den Dries and Macintyre \cite{CDM80, CDM81} described elementary invariants for those fields. This was inspired by the work of Ax on pseudo-finite fields. Among those invariants is the elementary theory of the absolute Galois group of those fields in a suitable omega-sorted language, called the \emph{inverse system of the absolute Galois group}. It was already clear to the authors of \cite{CDM80, CDM81} that this invariant is an essential tool for the study of PAC fields. The intuition that the model theoretic complexity of the theory of PAC fields is mainly controlled by the theory of its absolute Galois group has since then been confirmed by numerous results. For example, Chatzidakis \cite{Cha19} proved that if the inverse system of the absolute Galois group of a PAC field is NSOP$_n$ ($n>2$), then so is the theory of the field. Ramsey \cite{ramsey2018} proved the corresponding results for NTP$_1$ and NSOP$_1$. It is a fact that the inverse system of the absolute Galois group of a field $F$ is interpretable in the theory of the pair $(K,F)$ for any algebraically closed field $K$ extending $F$ (see \cite[Proposition 5.5]{chatzidakis2002}). This motivated our interest in the model-theoretic study of such pairs $(K,F)$. 

The model-theoretic study of pairs of fields goes back to Tarski when he raised in \cite{Tar51} the question of the decidability of the pair $(\R,\R\cap \Q^{\mathrm{alg}})$ (the reals with a predicate for the reals algebraic over $\Q$). The (positive) answer was given by Robinson in \cite{Rob59}, who gave a full set of axioms for the theories of $(\R,\R\cap \Q^{\mathrm{alg}})$ and $(\mathbb{C}, \Q^{\mathrm{alg}})$. The celebrated work of Morley and of Shelah in the 70{'}s created a growing interest in classification of first-order theories, and in particular of theories of fields and their expansions. It was known since the 80s that the theory of $(\mathbb{C}, \Q^{\mathrm{alg}})$ is stable\footnote{See the first sentence of \cite{Poizat1983}.} and Poizat \cite{Poizat1983} generalized this result to a more general context: he gave a criterion for the stability of special pairs of elementary substructures $N \succeq M$ (called ``belle paires"), under a strong stability assumption on the theory of $M$ (and $N$) called \emph{nfcp}, introduced by Keisler \cite{Kei67}. This was later generalized to the context of simple theories \cite{BENYAACOV2003235} with the notion of lovely pairs. Back to algebraically closed fields, Delon \cite{Delon2012} introduced a language for quantifier elimination for pairs of algebraically closed fields and pairs of algebraically closed valued fields. Recently, Martin-Pizarro and Ziegler \cite{MARTIN_PIZARRO_2020} proved that the theory of proper pairs of algebraically closed fields is equational, by a deep analysis of definable sets.

As was mentioned above, the main topic of this chapter is another generalization of pairs of algebraically closed fields, which are pairs $(K,F)$ where $F$ is an arbitrary field, perhaps with some extra structure (in a language extending the language of fields), and $K\supseteq  F$ is an algebraically closed infinite extension. An early result about this theory was given by Keisler \cite{Keisler1964}: if $F$ and $F'$ are two elementarily equivalent fields (not real-closed nor algebraically closed and without extra structure), then the pairs $(K,F)$ and $(K',F')$ are also elementarily equivalent, for any algebraically closed extensions $K\supsetneq eq  F$, $K'\supsetneq  F'$. In \cite{HKR18}, Hils, Kamensky and Rideau gave a quantifier elimination result for the theory of the pairs $(K,F)$, which we also obtain in \cref{quantifier elimination}. We were not aware of this result while writing the proof and we decided to keep our proof for completeness.

The purpose of this chapter is twofold. For one, we are interested in the basic properties of the theory of the pairs such as saturated models, completeness, quantifier elimination and model-completeness. For example, as we mentioned above we prove quantifier elimination for the theory of pairs $(K,F)$ (see \cref{quantifier elimination}) in a natural expansion of the language following Delon's approach \cite{Delon2012}. This allows us to isolate a condition implying the model-completeness of the theory of the pair $(K,F)$ which is weaker than the model completeness of the theory of $F$ (see \cref{model completeness}). Secondly, we prove preservation of several classification-theoretic properties: if the theory of $F$ is ($\omega$-/super) stable/NIP/simple/NSOP$_1$, then so is the theory of the pair $(K,F)$ (see \cref{omega-stable,superstable,stable,NIP,simple,nsop1}). In the case of NSOP$_1$, we also identify Kim-independence for algebraically closed sets (see \cref{kim independence}).

As immediate applications we conclude that the theory of a PAC field $F$ in the language of rings is NSOP$_1$ iff the theory of its Galois group is (see \cref{galois group nsop1}) and prove that when $F$ is pseudofinite in the language of rings the theory of the pair $(K,F)$ is simple. In addition, we consider the theory ACF$^I$ of a chain of algebraically closed fields ordered by some linear order $I$, and discuss its properties depending on the order type of $I$ (see \cref{tuples of acf order type}).
}{
In their study of pseudo-algebraically closed fields, or PAC fields (known at that time as regularly closed fields, for obvious reasons, see \cref{def PSF}) Cherlin, van den Dries and Macintyre \cite{CDM80, CDM81} described elementary invariants for those fields. This was inspired by the work of Ax on pseudo-finite fields. Among those invariants is the elementary theory of the absolute Galois group of those fields in a suitable omega-sorted language, called the \emph{inverse system of the absolute Galois group}. It was already clear to the authors of \cite{CDM80, CDM81} that this invariant is an essential tool for the study of PAC fields. The intuition that the model theoretic complexity of the theory of PAC fields is mainly controlled by the theory of its absolute Galois group was confirmed by numerous results since then. For example, Chatzidakis \cite{Cha19} proved that if the inverse system of the absolute Galois group of a PAC field is NSOP$_n$ ($n>2$), then so is the theory of the field. Ramsey \cite{ramsey2018} proved the corresponding results for NTP$_1$ and NSOP$_1$. It is a fact that the inverse system of the absolute Galois group of a field $F$ is interpretable in the theory of the pair $(K,F)$ for any algebraically closed field $K$ extending $F$ (see \cite[Proposition 5.5]{chatzidakis2002}). This motivated our interest in the model-theoretic study of such pairs $(K,F)$. 

The model-theoretic study of pairs of fields goes back to Tarski when he raised in \cite{Tar51} the question of the decidability of the pair $(\R,\R\cap \Q^{\mathrm{alg}})$ (the reals with a predicate for the reals algebraic over $\Q$). The (positive) answer was given by Robinson in \cite{Rob59}, who gave a full set of axioms for the theories of $(\R,\R\cap \Q^{\mathrm{alg}})$ and $(\mathbb{C}, \Q^{\mathrm{alg}})$. The celebrated work of Morley and of Shelah in the 70s created a growing interest in classification of first-order theories, and in particular of theories of fields and their expansions. It was known since the 80{'}s that the theory of $(\mathbb{C}, \Q^{\mathrm{alg}})$ is stable\footnote{See the first sentence of \cite{Poizat1983}.} and Poizat \cite{Poizat1983} generalized this result to a more general context: he gave a criterion for the stability of special pairs of elementary substructures $N \succeq M$ (called ``belle paires"), under a strong stability assumption on the theory of $M$ (and $N$) called \emph{nfcp}, introduced by Keisler \cite{Kei67}. This was later generalised to the context of simple theories \cite{BENYAACOV2003235} with the notion of lovely pairs. Back to algebraically closed fields, Delon \cite{Delon2012} introduced a language for quantifier elimination for proper pairs of algebraically closed fields (which are models of the theory of belles paires of algebraically closed fields) and proper pairs of algebraically closed valued fields. Recently, Martin-Pizarro and Ziegler \cite{MARTIN_PIZARRO_2020} proved that the theory of proper pairs of algebraically closed fields is equational, by a deep analysis of definable sets.

As was mentioned above, the main topic of this paper is another generalization of pairs of algebraically closed fields which are pairs $(K,F)$ where $F$ is an arbitrary field, perhaps with some extra structure (in a language extending the language of rings), and $K\supseteq  F$ is an algebraically closed field, such that the degree of $K$ over $F$ is infinite. An early result about this theory was given by Keisler \cite{Keisler1964}: if $F$ and $F'$ are two elementarily equivalent fields (not real-closed nor algebraically closed and without extra structure), then the pairs $(K,F)$ and $(K',F')$ are also elementarily equivalent, for any algebraically closed extensions $K\supsetneq  F$, $K'\supsetneq  F'$. In \cite{HKR18}, Hils, Kamensky and Rideau gave a quantifier elimination result for the theory of the pairs $(K,F)$, which we also obtain in \cref{quantifier elimination} (we became aware of their work only after we finished writing our proof and we decided to keep it for completeness).

The purpose of this paper is twofold: (1) investigate the basic logical properties of the theory of such pairs and (2) prove preservation results of several classification-theoretic properties. 

For (1), we discuss saturated models, completeness, quantifier elimination and model-completeness. For example, as we mentioned above we prove quantifier elimination for the theory of pairs $(K,F)$ (see \cref{quantifier elimination}) in a natural expansion of the language following Delon's approach \cite{Delon2012}. This allows us to isolate a condition implying the model-completeness of the theory of the pair $(K,F)$ which is weaker than the model completeness of the theory of $F$ (see \cref{model completeness}). For (2), we prove preservation of several classification-theoretic properties: if the theory of $F$ is ($\omega$-/super) stable/NIP/simple/NSOP$_1$, then so is the theory of the pair $(K,F)$ (see \cref{omega-stable,superstable,stable,NIP,simple,nsop1}). In the case of NSOP$_1$, we also identify Kim-independence for algebraically closed sets (see \cref{kim independence}).

As immediate applications we conclude that:
\begin{enumerate}
    \item The theory of a PAC field $F$ in the language of rings is NSOP$_1$ if and only if the theory of its Galois group is (see \cref{galois group nsop1}).
    \item When $F$ is pseudofinite in the language of rings, then the theory of the pair $(K,F)$ is simple.   
\end{enumerate}
  In addition, we consider the theory ACF$^I$ of a chain of algebraically closed fields ordered by some linear order $I$, and discuss its properties depending on the order type of $I$ (see \cref{tuples of acf order type}).\\

\noindent \textbf{Acknowledgement.} The authors would like to thank Zo\'e Chatzidakis for her useful comments and give a special thanks to Nick Ramsey for valuable discussions and ideas in this project. We would also like to thank Anand Pillay for his comments leading us to \cref{strongly minimal stable} and \cref{question strongly minimal}.
We would also like to thank the anonymous referee for their careful reading and their comments. 
}

\section{Preliminaries} \label{subfield preliminaries}

In this section we present common definitions and results from fields and model theory. 
We will start by setting up some basic notation for the whole paper.
\begin{notation}
  Whenever $A$ is a field, let $\overline{A}$ be its algebraic closure.
  Whenever $A$ and $B$ are subfields of a larger field, let $A.B$ be their field compositum.
  If $A$ is a field and $S$ is a set, then let $A(S)$ be the field extension of $A$ by the elements of S.
  Say that the set $S$ is algebraically independent over $A$ if each element $s\in S$ is algebraically independent over $A(S\setminus\set{s})$.
  If $R$ is a sub-ring of a larger field, then denote by $\fr(R)$ the field generated by $R$.
  Unless specified otherwise, all the fields will be subfields of a large algebraically closed field.
\end{notation}

\subsection{Linear disjointness}

\begin{definition}
  Let $A$, $B$ and $C$ be fields with $C\subseteq A\cap B$.
  \begin{enumerate}
      \item Say that $A$ is \emph{linearly disjoint} from $B$ over $C$ if whenever $a_0,\dots ,a_{n-1}\in A$ are linearly independent over $C$ they are also linearly independent over $B$. 
      Denote this by $A\forkindep_C^l B$.
      \item Say that $A$ is \emph{algebraically disjoint} from $B$ over $C$ if whenever $a_0,\dots ,a_{n-1} \in A$ are algebraically independent over $C$, then they are also algebraically independent over $B$. 
      This is the same as the non-forking independence in $\ACF$, which we will denote $A\forkindep^\ACF_C B$.
  \end{enumerate}
\end{definition}

\begin{fact}[{\cite[Proposition 20.2]{morandi1996field}}] \label{linear disjointedness tensor products}
  Let $A$, $B$ and $C$ be fields with $C\subseteq A\cap B$.
  Construct a map $A\otimes_C B\to A[B]$ by mapping $a\otimes b\mapsto ab$.
  This map is an isomorphism iff $A\forkindep^l_C B$.
\end{fact}

\begin{fact} \label{linear disjointedness properties}
  The following is a list of useful model theoretic properties that $\forkindep^l$ has inside $\ACF$.
  Let $A$, $B$, $C$, $D$, $A'$, $B'$ and $C'$ be fields with $C\subseteq A\cap B$, $C'\subseteq A'\cap B'$ and $B\subseteq D$.
  \begin{itemize}
    \item (Invariance) if $ABC\equiv A'B'C'$ and $A\forkindep^l_C B$, then $A'\forkindep^l_{C'} B'$.
    \item (Monotonicity) if $A\forkindep^l_C D$, then $A\forkindep^l_C B$.
    \item (Base monotonicity) if $A\forkindep^l_C D$, then $A.B\forkindep^l_{B} D$.
    \item (Transitivity) if $A\forkindep^l_{C} B$ and $A.B\forkindep^l_{B} D$, then  $A\forkindep^l_C D$.
    \item (Symmetry) if $A\forkindep^l_C B$, then $B\forkindep^l_C A$.
    \item (Stationarity) if $A\equiv_C A'$ and $A\forkindep^l_C B$, $A'\forkindep^l_C B$, then $A\equiv_{B} A'$.
    \item (Local character) for a finite tuple $a$, there exists a countable subfield $B_0\subseteq B$, such that $B_0(a)\forkindep^l_{B_0} B$.
  \end{itemize}
\end{fact}

\begin{proof}
  Invariance is trivial.
  Proofs for monotonicity, base monotonicity and transitivity can be found in  \cite[Lemma 2.5.3]{fried2008field}, symmetry is proven in \cite[Lemma 2.5.1]{fried2008field}.
  Stationarity follows directly from \cref{linear disjointedness tensor products} and quantifier elimination in ACF.
  
  Local character follows from \cite[Theorem III.7, Proposition III.6 and Theorem III.8]{lang2019introduction}, by setting $B_0$ to be the field of definition of the locus of $a$ over $B$.
  This gives an even stronger result, as $B_0$ is finitely generated and not merely countable.
  For a more direct proof of local character, see \cref{shorter local character}.
\end{proof}

\begin{corollary} \label{isomorphism amalgamation}
  Let $A_0$, $B_0$, $C_0$, $A_1$, $B_1$ and $C_1$ be fields with $C_0\subseteq A_0\cap B_0$, $C_1\subseteq A_1\cap B_1$, such that $A_0\forkindep^l_{C_0} B_0$, $A_1\forkindep^l_{C_1} B_1$.
  Suppose there are isomorphism $f:A_0\to A_1$, $g:B_0\to B_1$ such that $f|_{C_0}=g|_{C_0}$.
  Then there is a unique isomorphism $F:A_0.B_0\to A_1.B_1$ such that $F|_{A_0}=f$, $F|_{B_0}=g$.
\end{corollary}

\begin{proof}
  Consider $A_0$, $A_1$, $B_0$ and $B_1$ as tuples, such that $f$ and $g$ match the tuples.
  Extend $g$ to an automorphism $\sigma$ arbitrarily.
  From invariance, by applying $\sigma$ to $A_0\forkindep^l_{C_0} B_0$, we get $\sigma(A_0)\forkindep^l_{C_1} B_1$.
  From stationarity $\sigma(A_0)\equiv_{B_1} A_1$, let $\tau$ be an automorphism witnessing the equivalence.
  Let $F=(\tau\circ \sigma)|_{A_0.B_0}$, we have $F(A_0)=\tau(\sigma(A_0))=A_1$ and $F(B_0)=\tau(\sigma(B_0))=\tau(B_1)=B_1$ as tuples.
  In particular, $F:A_0.B_0\to A_1.B_1$ is an isomorphism, and from the way we chose the tuples $F|_{A_0}=f$ and $F|_{B_0}=g$.
\end{proof}

\begin{definition}
  A field extension $A\subseteq B$ is called:
  \begin{itemize}
      \item \emph{regular} if $\overline{A}\forkindep^l_A B$,
      \item \emph{separable} if $A^{1/p}\forkindep^l_A B$, where $p=\characteristic(A)>0$ and $A^{1/p}$ is the field of $p$-th roots of all elements in $A$ (if $\characteristic(A)=0$, then all extensions are separable), and
      \item \emph{relatively algebraically closed} if $\overline{A}\cap B=A$.
  \end{itemize}
\end{definition}

\begin{fact} \label{fact regular}
    Suppose $A\subseteq B$ is a field extension.
    \begin{enumerate}
        \item \cite[Lemma 2.6.4]{fried2008field} The extension $A\subseteq B$ is regular iff it is separable and relatively algebraically closed.
        \item \cite[Lemma 2.6.7]{fried2008field} If the extension $A\subseteq B$ is regular and $C$ is a field extending $A$ such that $B\forkindep^\ACF_A C$, then $B\forkindep^l_A C$.
    \end{enumerate}
\end{fact}

\begin{lemma} \label{regular extension isomorphism}
  If $A\subseteq B$ is a regular field extension and $\sigma:B\to B'$ is an isomorphism of fields,
  then $\sigma(A)\subseteq B'$ is regular.
\end{lemma}

\begin{proof}
  We can extend $\sigma$ to the algebraic closure, $\tilde{\sigma}: \overline{B}\to \overline{B'}$.
  From $\overline{A}\forkindep^l_A B$ we get by invariance $\tilde{\sigma}(\overline{A})\forkindep^l_{\sigma(A)} B'$.
  But $\tilde{\sigma}(\overline{A})=\overline{\sigma(A)}$, so we have $\overline{\sigma(A)}\forkindep^l_{\sigma(A)} B'$ as needed.
\end{proof}

\begin{lemma} \label{regular algebraically independent}
  If $A\subseteq B$ is a regular field extension and $S$ is a set algebraically independent over $B$, then $\overline{A(S)}\forkindep^l_A B$.
\end{lemma}

\begin{proof}
  As $S$ is algebraically independent over $B$, we have $\overline{A(S)} \forkindep^\ACF_A B$.
  By \cref{fact regular}(2), $\overline{A(S)}\forkindep^l_A B$.
\end{proof}

\subsection{Language of regular extensions}
In \cite{MACINTYRE2008}, Macintyre defines relations in the language of rings that are preserved in a field extension iff it is regular. 
We will present those relations, and use them to expand a theory of fields\footnote{By a \emph{theory of fields}, we mean a theory in a language expanding the language of rings which contains all the fields axioms.} in such a way that the models are the same but for any two models $M,N$, $N$ extends $M$ iff it is a regular field extension.

\begin{fact}[{\cite[\S 4.7]{MACINTYRE2008}}] \label{RAC and separability predicates}
  Let $A\subseteq B$ be a field extension.
  \begin{enumerate}
    \item The extension is relatively algebraically closed iff it preserves the relations $\mathrm{Sol}_n(x_0,\dots ,x_{n-1})=\exists y(x_0+x_1y+\dots +x_{n-1}y^{n-1}+y^n=0)$ for $n\ge 1$.
    \item For $p=\characteristic(A)$, the extension is separable iff it preserves the relations $D_{n,p}(x_0,\dots ,x_{n-1})=\exists y_0,\dots ,y_{n-1}(y_0^px_0+\dots +y_{n-1}^px_{n-1}=0)$ for $n\ge 1$ (note that if $p=0$, $D_{n,p}$ is quantifier-free definable).
  \end{enumerate}
\end{fact}

\begin{corollary} \label{elementary extension is regular}
  Suppose $M$ and $N$ are fields.
  If $M\preceq N$, then $M\subseteq N$ is a regular extension.
\end{corollary}

\begin{proof}
  The fact that $M\preceq N$ implies in particular that $M\subseteq N$ is a field extension that preserves $\mathrm{Sol}_n$ and $D_{n,p}$ ($p=\characteristic(A)$).
  By \cref{RAC and separability predicates} the extension $M\subseteq N$ is relatively algebraically closed and separable, so by \cref{fact regular}(1) it is a regular extension.
\end{proof}

\begin{definition} \label{language of regular extensions}
  Let $T$ be a theory of fields in a language $L$ expanding the language of rings.
  Define $L_{\mathrm{reg}}=L\cup\set{\mathrm{Sol}_n}_{n\ge 1}\cup\set{\tilde{D}_{n,p}}_{n\ge1,p\in\mathrm{Primes}\cup\{0\}}$, where $\mathrm{Sol}_n$, $\tilde{D}_{n,p}$ are $n$-ary relations, and extend $T$ to $T_{\mathrm{reg}}$ in $L_{\mathrm{reg}}$ by defining $\mathrm{Sol_n}$ as above and defining
  $$\tilde{D}_{n,p}=D_{n,p}\land (\underbrace{1+\dots +1}_p=0).$$
\end{definition}

\begin{lemma} \label{substructure iff regular}
  Let $T$ be a theory of fields and let $Q,R\vDash T$ with $Q\subseteq R$ a substructure.
  By adding definable relations, $Q$ and $R$ can be expanded to models of $T_{\mathrm{reg}}$.
  Then $Q$ is an $L_{\mathrm{reg}}$-substructure of $R$ iff $Q\subseteq R$ is a regular field extension.
\end{lemma}

\begin{proof}
  Let $p=\characteristic(Q)$.
  Note that by Facts \ref{fact regular} and \ref{RAC and separability predicates}, it is enough to prove that $Q$ is an $L_{\mathrm{reg}}$-substructure of $R$ iff the extension $Q\subseteq R$ preserves $\mathrm{Sol}_n$ and $D_{n,p}$ for all $n$.
  Indeed, this equivalence holds because $\tilde{D}_{n,p}$ is equivalent to $D_{n,p}$ and $\tilde{D}_{n,q}$ is trivially false for any prime $q\ne p$.
\end{proof}

\subsection{NSOP$_1$}
In this subsection we will review the definition and basic properties of NSOP$_1$ theories.

We will work in a monster model $\M$ (large, saturated) of a complete theory $T$.

\begin{definition}
    A formula $\phi(x;y)$ has SOP$_1$ if there is a tree of tuples $(b_\eta)_{\eta\in 2^{<\omega}}$ such that
    \begin{itemize}
       \item for all $\eta\in 2^\omega$, $\set{\phi(x;b_{\eta|\alpha})\mid \alpha<\omega}$ is consistent,
       \item for all $\eta\in 2^{<\omega}$, if $\nu\unrhd\eta\frown \vect{0}$, then $\set{\phi(x;b_\nu),\phi(b;a_{\eta\frown\vect{1}}}$ is inconsistent.
    \end{itemize}
    We say that a theory $T$ is SOP$_1$ if some formula has SOP$_1$ modulo $T$.
    Otherwise, $T$ is NSOP$_1$.
\end{definition}

\begin{definition}
  Let $A$ be a set and $a$ and $b$ tuples, say that $a$ is \emph{coheir independent} of $b$ over $A$ if the type $\tp(a/Ab)$ is finitely satisfiable in $A$, and denote $a\forkindep^u_A b$.
  A sequence $(a_i)_{i\in I}$ is an $A$-indiscernible coheir sequence if it is $A$-indiscernible and $a_i\forkindep^u_A a_{<i}$
\end{definition}

Using coheir-independence, we can use a different criterion for NSOP$_1$, due to \cite[Theorem 5.7]{Chernikov2016}.

 \begin{fact}[Weak independent amalgamation]\label{WIA}
    The theory $T$ is NSOP$_1$ iff given any model $M\vDash T$ and tuples $a_0b_0\equiv_M a_1b_1$ such that $b_1\forkindep^u_M b_0$ and $b_i\forkindep^u_M a_i$ for $i=0,1$, there exists $a$ such that $ab_0\equiv_M ab_1\equiv_M a_0b_0$.
\end{fact}

Kim-dividing, and its extension Kim-forking, were defined in \cite{Kaplan2017}, over arbitrary sets. 
For our purposes we will give a simplified definition, which we will call Kim$^u$-dividing, and define it only over models.
\begin{definition} \label{def kim u dividing}
    A formula $\phi(x,b)$ \emph{Kim$^u$-divides} over a model $M$ if there exists an $M$-indiscernible coheir sequence $(b_i)_{i<\omega}$ with $b\equiv_M b_i$, such that $\set{\phi(x,b_i)}_{i<\omega}$ is inconsistent.
    A formula \emph{Kim$^u$-forks} over $M$ if it implies a disjunction of Kim$^u$-dividing formulas over $M$.
    
    A type Kim$^u$-divides (Kim$^u$-forks) over $M$ if it implies a Kim$^u$-dividing (Kim$^u$-forking) formula over $M$.
    Denote $a\forkindep^K_M b$ when the type $\mathrm{tp}(a/Mb)$ does not Kim$^u$-fork over $M$.
\end{definition}

\begin{remarkcnt} \label{kimu equals kim}
In this definition, $(b_i)_{i<\omega}$ is a Morley sequence in a restriction of a global coheir type.
In the original definition of Kim-dividing, the global coheir type is replaced with a global invariant type.
By Kim's lemma for Kim-dividing \cite[Theorem 3.16]{Kaplan2017}, those definitions are equivalent for NSOP$_1$ theories.
\end{remarkcnt}

\begin{remarkcnt} \label{kimu dividing}
The type $\tp(a/Mb)$ does not Kim$^u$-divide over $M$ iff for every $M$-indiscernible coheir sequence $(b_i)_{i<\omega}$ with $b\equiv_M b_i$, there exists $a'$ such that $ab\equiv_M a'b_i$ for every $i<\omega$.
\end{remarkcnt}

\begin{fact}\label{kim independence in nsop1}
    Suppose $T$ is NSOP$_1$, then
    \begin{enumerate}
        \item \cite[Theorem 3.16]{Kaplan2017} If $\phi(x,b)$ Kim-divides over $M\vDash T$, then for \emph{every} $M$-indiscernible coheir sequence $(b_i)_{i<\omega}$ with $b\equiv_M b_i$, $\set{\phi(x,b_i)}_{i<\omega}$ is inconsistent.
        \item \cite[Proposition 3.19]{Kaplan2017} Kim-dividing is equivalent to Kim-forking over models.
        \item \cite[Theorem 5.16]{Kaplan2017} $\forkindep^K$ is symmetric over models.
        \item \cite[Corolary 5.17]{Kaplan2017} Let $M\vDash T$, $a\forkindep^K_M b\iff \acl(a)\forkindep^K_M b\iff a\forkindep^K_M \acl(b)$.
        \item \cite[Proposition 8.8]{Kaplan2017} $T$ is simple iff $\forkindep^K$ satisfies base monotonicity over models: if $M,N\vDash T$ and $M\subseteq N$, then $a\forkindep^K_M Nb$ implies $a\forkindep^K_N b$.
        \item \cite[Proposition 8.4]{Kaplan2017} $T$ is simple iff $\forkindep^K=\forkindep^f$ over models.
    \end{enumerate}
\end{fact}

\section{Basic properties of ACF\texorpdfstring{$_T$}{T}}\label{subfield basic properties of ACFT}

In this section we will define and study the basic properties of $\ACF_T$, the theory of algebraically closed fields with a distinguished subfield (in an arbitrary language).
We will also consider expansions of the theory by definable relations and functions, that Delon defined to study pairs of $\ACF$ in \cite{Delon2012}.

\subsection{Delon's language}

\begin{definition} \label{delons language}
  Let $T$ be a theory of fields (not necessarily complete), in a language expanding the language of rings $L\supseteq L_{\mathrm{rings}}$.
  Expand $L$ to the language $L^P=L\cup\set{P}$, with $P$ a unitary predicate, and expand $\ACF$ to $\ACF_T$ in the language $L^P$ by adding the following axioms:
  \begin{enumerate}
    \item $P$ is a subfield of the universe, i.e. $P$ is closed under the ring operations (and contains $0,1$).
    \item $P$ is a model of $T$.
           This can be achieved by taking all the axioms of $T$ and restricting the quantifiers to be over $P$ (see \cref{corresponding formula}).
    \item For every $n$-ary function symbol $f\in L\setminus L_{\mathrm{rings}}$, if $x_0,\dots x_{n-1}\in P$, then $f(x_0,\dots ,x_{n-1})\in P$.
          Else, if some $x_i\notin P$, then we do not care about the value of $f(x_0,\dots ,x_{n-1})$, and we can set it arbitrarily to $0$.
    \item For every $n$-ary relation symbol $R\in L$ (equivalently $R\in L\setminus L_{\mathrm{rings}}$ as $L_{\mathrm{rings}}$ does not have any relation symbols), if some $x_i\notin P$, then $\lnot R(x_0,\dots ,x_{n-1})$.
          That is, $R\subseteq P^n$.
    \item The degree of the field extension of the universe over $P$ is infinite, i.e. the universe has infinite dimension as a vector space over $P$.
    By the Artin-Schreier theorem \cite{ArtinSchreier1926}, it is enough to assert that the degree is at least $3$.
  \end{enumerate}
\end{definition}

\begin{remarkcnt}
  The assumption that the degree of the universe over $P$ is infinite, that is, for $M\vDash \ACF_T$, $[M:P_M]=\infty$, always holds when models of $T$ are not algebraically closed or real closed, because in that case $[\overline{P_M}:P_M]=\infty$.
  When models of $T$ are algebraically closed, it simply means that $M\ne P_M$, i.e. $(M,P_M)$ is a proper pair.
  The only case excluded is when models of $T$ are real closed and $M=\overline{P_M}$, but then $(\overline{P_M},P_M)$ is definable in $P_M$.
\end{remarkcnt}

\begin{definition}
  Let $T$, $L$ be as above. Consider the following definable relations and functions over $\ACF_T$:
  \begin{itemize}
    \item For $n\ge 1$, define the $n$-ary relation $l_n$ by $l_n(x_0,\dots ,x_{n-1})$ iff $x_0,\dots ,x_{n-1}$ are linearly independent over $P$.
    \item For $n\ge 1$, suppose we have $l_n(x_0,\dots ,x_{n-1})$ and $\lnot l_{n+1}(x_0,\dots ,x_n)$.
    That is, $x_0,\dots ,x_{n-1}$ are linearly independent over $P$ and $x_n$ is in their span over $P$.
    Then there are unique $y_i\in P$ such that $x_n=y_0x_0+\dots + y_{n-1}x_{n-1}$.
    Define the $n+1$-ary function $f_{n,i}$ by $f_{n,i}(x_n;x_0,\dots,x_{n-1})=y_i$.
    If $x_0,\dots,x_n$ do not satisfy this condition, then we do not care about the value of $f_{n,i}(x_n;x_0,\dots,x_{n-1})$ and can set it arbitrarily to $0$.
  \end{itemize}
  Expand $\ACF_T$ to $\ACF_T^{ld}$ in the language $L^{ld}=L^P\cup \set{l_n}_{n\ge 1}$, by defining $l_n$ as above.
  Expand $\ACF_T^{ld}$ to $\ACF_T^{f}$ in the language $L^f=L^{ld}\cup \set{f_{n,i}}_{n>i\ge 0}$, by defining $f_{n,i}$ as above.
\end{definition}

\begin{notation}
  If $M\models\ACF_T$, then let $P_M$ be the predicate $P$ in $M$ with the associated $L$-structure.
  If $A\subseteq M$ is a subset, then let $P_A=P_M\cap A$.
  This notation is used instead of the usual $P(M)$ and $P(A)$, because the notation $P(A)$ is reserved for the field extension of $P$ by $A$.
\end{notation}

\begin{definition} \label{def bounded}
  Call a formula $\phi(x)\in L^P$ \emph{bounded} if every quantifier in $\phi$ is over $P$.
\end{definition}

\begin{remarkcnt} \label{corresponding formula}
    For a formula $\phi(x)\in L$ there is a corresponding bounded formula $\phi^P(x)\in L^P$ created by restricting every quantifier to be over $P$ and asserting $x\in P$. 
    For $M\vDash \ACF_T$, we have $\phi^P(M)=\phi(P_M)$.
\end{remarkcnt}
\subsection{Substructures and isomorphisms}
\begin{lemma} \label{substructure}
  Let $M\vDash \ACF_T^{f}$ and $A\subseteq M$ a subset.
  Then $A$ is an $L^f$-substructure iff $P_A\subseteq P_M$ is an $L$-substructure, $A$ is a subring, $P_A$ is a subfield and $\fr(A)\forkindep^l_{P_A} P_M$.
\end{lemma}

\begin{proof}
  Suppose $A\subseteq M$ is an $L^f$-substructure.
  We get that $P_A\subseteq P_M$ is an $L$-substructure, because for any function symbol $f\in L$ and $\overline{a}\in P_A$, $f(\overline{a})\in A$ as $A\subseteq M$ is a substructure, and also $f(\overline{a})\in P_M$ because of the axioms of $\ACF_T$, so $f(\overline{a})\in A\cap P_M=P_A$.
  It is clear that $A$ is a subring, and so is $P_A$, but for every $0\ne a\in P_A$, $a^{-1}=f_{1,0}(1;a)\in P_A$, so $P_A$ is also a subfield.
  By \cite[Chapter III, Criterion 1]{lang2019introduction}, to prove that $\fr(A)\forkindep^l_{P_A} P_M$, it is enough to show that if $a_0,\dots,a_{n-1}\in A$ are linearly dependent over $P_M$, then they are linearly dependent over $P_A$. 
  Suppose $a_0,\dots ,a_{n-1}\in A$  are linearly dependent over $P_M$.
  If $a_0=0$, then the tuple is trivially linearly dependent over $P_A$.
  Else, there is some maximal $1\le k<n$ such that $a_0,\dots ,a_{k-1}$ are linearly independent over $P_M$, so we have $\vDash l_k(a_0,..,a_{k-1})$ and $\models\lnot l_{k+1}(a_0,\dots ,a_k)$.
  Hence we can look at $p_i=f_{k,i}(a_k;a_0,\dots ,a_{k-1})\in P_M$, which give us $a_k=p_0a_0+\dots +p_{k-1}a_{k-1}$.
  Because $A$ is a substructure, $p_i\in A$, so $p_i\in P_A$.
  Thus, $a_0,\dots ,a_{n-1}$ are linearly dependent over $P_A$.

  In the other direction, suppose $A$ is a subring, $P_A$ is a subfield, $P_A\subseteq P_M$ is an $L$-substructure and $\fr(A)\forkindep^l_{P_A} P_M$.
  It follows that $\fr(A)\cap P_M = P_A$, and in particular $A\cap P_M = P_A$.
  For any function symbol $f\in L\setminus L_{\mathrm{rings}}$ and $a_0,\dots ,a_{n-1}\in A$, if $a_0,\dots ,a_{n-1}\in P_A$, then $f(a_0,\dots ,a_{n-1})\in P_A$ as $P_A\subseteq P_M$ is a substructure, and else we defined $f(a_0,\dots ,a_{n-1})=0\in A$.
  It remains to check that $A$ is closed under $f_{n,i}$.
  Let $a_0,\dots ,a_n\in A$ and suppose $\vDash l_n(a_0,\dots ,a_{n-1})$, $\models\lnot l_{n+1}(a_0,\dots ,a_n)$.
  Let $p_i=f_{n,i}(a_n;a_0,\dots ,a_{n-1})$, that is $p_i\in P_M$ and $a_n=p_0a_0+\dots +p_{n-1}a_{n-1}$.
  We know that $a_0,\dots ,a_n$ are linearly dependent over $P_M$, so by $\fr(A)\forkindep^l_{P_A} P_M$ they are linearly dependent over $P_A$.
  However, $a_0,\dots ,a_{n-1}$ must be linearly independent over $P_A$, as they are linearly independent over $P_M$, so $a_n$ can be written as a linear combination of $a_0,\dots ,a_{n-1}$ over $P_A$.
  This linear combination is in particular over $P_M$, but $a_n=p_0a_0+\dots +p_{n-1}a_{n-1}$ is the unique linear combination over $P_M$, so we must have $p_0,\dots ,p_{n-1}\in P_A$, as needed.
\end{proof}

\begin{corollary} \label{field of fractions substructure}
    If $M\vDash \ACF_T^f$ and $A\subseteq M$ is an $L^f$-substructure, then $\fr(A)\subseteq M$ is an $L^f$-substructure with $P_{\fr(A)}=P_A$.
\end{corollary}

\begin{proof}
    \cref{substructure} implies that $\fr(A)\forkindep^l_{P_A} P_M$, and in particular $P_{\fr(A)}=P_M\cap \fr(A)=P_A$.
    Thus, $P_{\fr(A)}\subseteq P_M$ is a subfield and an $L$-substructure, $\fr(A)$ is a subring (even subfield) and $\fr(A)\forkindep^l_{P_{\fr(A)}} P_M$, so by \cref{substructure} $\fr(A)\subseteq M$ is an $L^f$-substructure.
\end{proof}

\begin{lemma} \label{isomorphism}
  Let $M,N\vDash \ACF_T^f$ and let $A\subseteq M$, $B\subseteq N$ be $L^f$-substructures.
  A map $\sigma:A\to B$ is an $L^f$-isomorphism iff $\sigma$ is an isomorphism of rings such that $\sigma(P_A)=P_B$ and $\sigma|_{P_A}:P_A\to P_B$ is an $L$-isomorphism.
\end{lemma}

\begin{proof}
    If $\sigma$ is an $L^f$ isomorphism, then it is clearly an isomorphism of rings, $\sigma(P_A)=P_B$ because $\sigma$ preserves $P$ and $\sigma|_{P_A}:P_A\to P_B$ is an $L$-isomorphism because $L^f$ expands $L$ on $P$.
    For the other direction, we need to show that $\sigma$ preserves $l_n$, $f_{n,i}$.
    Let $a_0,\dots ,a_{n-1}\in A$ with $\vDash l_n(a_0,\dots ,a_{n-1})$.
    Suppose we have $\models\lnot l_n(\sigma(a_0),\dots ,\sigma(a_{n-1}))$, i.e. $\sigma(a_0),\dots ,\sigma(a_{n-1})$ are linearly dependent over $P_N$.
    \cref{substructure} implies that $\fr(B)\forkindep^l_{P_B} P_N$, so $\sigma(a_0),\dots ,\sigma(a_{n-1})$ are also linearly dependent over $P_B$.
    Thus, there are $q_0,\dots ,q_{n-1}\in P_B$, not all zero, such that $q_0\sigma(a_0)+\dots +q_{n-1}\sigma(a_{n-1})=0$.
    By applying $\sigma^{-1}$ we get $\sigma^{-1}(q_0)a_0+\dots +\sigma^{-1}(q_{n-1})a_{n-1}=0$, however $\sigma^{-1}(q_0),\dots ,\sigma^{-1}(q_{n-1})\in P_A$, in contradiction to $\vDash l_n(a_0,\dots ,a_{n-1})$.
    The other direction follows from symmetry.
    Now suppose we have $a_0,\dots ,a_n\in A$ with $\vDash l_n(a_0,\dots ,a_{n-1})$ and $\models\lnot l_{n+1}(a_0,\dots ,a_n)$.
    By the first part, we also have $\vDash l_n(\sigma(a_0),\dots ,\sigma(a_{n-1}))$ and $\models\lnot l_{n+1}(\sigma(a_0),\dots ,\sigma(a_n))$.
    Let $p_i=f_{n,i}(a_n;a_0,\dots ,a_{n-1})\in P_A$, $a_n=p_0a_0+\dots +p_{n-1}a_{n-1}$.
    Apply $\sigma$ to get $\sigma(a_n)=\sigma(p_0)\sigma(a_0)+\dots +\sigma(p_{n-1})\sigma(a_{n-1})$, but $\sigma(p_0),\dots ,\sigma(p_{n-1})\in P_B$, so by uniqueness $\sigma(p_i)=f_{n,i}(\sigma(a_n);\sigma(a_0),\dots ,\sigma(a_{n-1})$.
\end{proof}

\begin{lemma} \label{submodel}
    Let $M,N\vDash \ACF_T$. 
    By adding definable relations and functions, $M$ and $N$ can be expanded to models of $\ACF_T^{ld}$, $\ACF_T^f$.
    With those expansions, the following are equivalent:
    \begin{enumerate}
        \item $M\subseteq N$ is an $L^f$-substructure.
        \item $M\subseteq N$ is an $L^{ld}$-substructure.
        \item $M\subseteq N$ is a subfield, $P_M\subseteq P_N$ is an $L$-substructure and $M\forkindep^l_{P_M} P_N$.
    \end{enumerate}
\end{lemma}

\begin{proof}
    $1\implies 2$: 
    $L^{ld}$ is a restriction of $L^f$.

    $2\implies 3$:
    It is clear that $M\subseteq N$ is a subfield and $P_M\subseteq P_N$ as sets.
    For every quantifier free formula $\phi(\overline{x})\in L$ and $\overline{a}\in P_M$, $P_M\vDash \phi(\overline{a})\iff M\vDash \phi(\overline{a})\land \overline{a}\in P\iff N\vDash \phi(\overline{a})\land \overline{a}\in P\iff P_N\models\phi(\overline{a})$, so $P_M$ is an $L$-substructure of $P_N$.
    Let $a_0,\dots,a_{n-1}\in M$ be linearly independent over $P_M$, $M\vDash l_n(a_0,\dots,a_{n-1})\implies N\vDash l_n(a_0,\dots,a_{n-1})$, so $a_0,\dots,a_{n-1}$ are linearly independent over $P_N$.
    Thus, $M\forkindep^l_{P_M} P_N$.

    $3\implies 1$:
    Let $M'$ be the $L^f$-structure with the same underlying set as $M$, but with structure induced as a subset of $N$.
    Note that $M'\subseteq N$ is really an $L^f$-substructure, from \cref{substructure}.
    To prove that $M$ is an $L^f$-substructure of $N$, we need to show that $M$ and $M'$ have the same structure, that is that the identity map $id:M\to M'$ is an $L^f$-isomorphism.
    We know that $M$ is a subfield of $N$, so $id:M\to M'$ is a field isomorphism. 
    From $M\forkindep^l_{P_M} P_N$ we get that $P_{M'}=M\cap P_N=P_M$ and $P_M$ is an $L$-substructure of $P_N$, so $id|_{P_M}:P_M\to P_{M'}$ is an $L$-isomorphism.
    \cref{isomorphism} implies that $id$ is an $L^f$-isomorphism.
\end{proof}

\subsection{Saturated models}
We will study saturated models of ACF$_T$.
Note that $\kappa$-saturated models of $\ACF_T$ are the same as $\kappa$-saturated models of $\ACF_T^{ld}$ or $\ACF_T^f$, because $\set{l_n}_{n>1}$ and $\set{f_{n,i}}_{n>i>0}$ are definable in $\ACF_T$.
A full characterization of $\kappa$-saturated models will be given in \cref{saturated}.

\begin{lemma} \label{saturation P}
  If $M\vDash \ACF_T$ is $\kappa$-saturated, then $P_M$ is a $\kappa$-saturated model of $T$.
\end{lemma}

\begin{proof}
  Follows from \cref{corresponding formula}, by relativizing each formula in the type we wish to realize to $P$.
\end{proof}

For the next result, we will need the following algebraic technical lemma, whose proof is left as an exercise to the reader.

\begin{fact} \label{transcendental degree}
 Suppose $F$ is a field and $t$ is transcendental over $F$. For every $n$, $[F(t):F(t^n)]=n$.
\end{fact}


\begin{lemma} \label{saturation trdeg}
  If $M\vDash \ACF_T$ is $\kappa$-saturated, then $\trdeg(M/P_M)\ge\kappa$.
\end{lemma}

\begin{proof}
  Let $S\subseteq M$ be an algebraically independent set over $P_M$. Suppose $\abs{S}<\kappa$, we want to prove that there is some $a\in M$ such that $a\notin \overline{P_M(S)}$.
  Consider the partial type over $S$
  $$\Sigma(x)=\set{\forall \bar{y}\in P\ (q(x,\bar{y})=0 \to\forall x' q(x',\bar{y})= 0)\mid q(x,\bar{y})\in Q[x,\bar{y},S]}$$
  where $Q$ is the prime field ($\F_p$ or $\Q$), $x$ is a single variable and $\bar{y}$ is a tuple of variables.
  Let $\Sigma_n(x)$ contain all formulas in $\Sigma(x)$ where the degree of $q(x,\bar{y})$ in $x$ is $\le n$.
  We will show that $a\vDash \Sigma_n(x)$ iff $[P_M(S,a):P_M(S)]>n$ and that $\Sigma_n(x)$ is satisfiable in $M$.
  From compactness and saturation ($\abs{S}<\kappa$), we will get that $\Sigma(x)$ is satisfied by some $a\in M$.
  But then $[P_M(S,a):P_M(S)]>n$ for all $n$, so $a\notin\overline{P_M(S)}$.
  
  Suppose $a\vDash \Sigma_n(x)$. 
  If $[P_M(S,a):P_M(S)]\le n$, then there is some non-zero polynomial $r(x)\in P_M(S)[x]$ of degree $\le n$ such that $r(a)=0$.
  The coefficients of $r(x)$ are rational functions in $S$ over $P_M$.
  By multiplying by the denominators, we can assume the coefficients are polynomials in $S$ and $P_M$, so $r(x)=q(x,\bar{p})$ for $q(x,\bar{y})\in Q[x,\bar{y},S]$ and $\bar{p}\in P_M$.
  However, because $q(a,\overline{p})=r(a)=0$, we get from $a\vDash \Sigma_n(x)$ that $r(x)$ is constant zero.
  
  Now suppose $[P_M(S,a):P_M(S)]>n$.
  Let $q(x,\bar{y})\in Q[x,\bar{y},S]$ of degree $\le n$ in $x$ and $\bar{p}\in P_M$, such that $q(a,\bar{p})=0$.
  The polynomial $q(x,\bar{p})$ is over $P_M(S)$, has degree $\le n$ and has $a$ as root, but $[P_M(S,a):P_M(S)]>n$, so $q(x,\bar{p})$ must be constant zero.
  Hence $a\vDash \Sigma_n(x)$.

  To prove that $\Sigma_n(x)$ is satisfiable for every $n$, we need to prove that there is some $a\in M$ such that $[P_M(S,a):P_M(S)]>n$.
  Split into three cases.
  \begin{enumerate}
    \item $S=\emptyset$, $M\ne\overline{P_M}$:
          Take some $a\in M\setminus \overline{P_M}$ and we are done.
    \item $S=\emptyset$, $M=\overline{P_M}$:
          The axioms of $\ACF_T$ (\cref{delons language}) imply that $[\overline{P_M}:P_M]=\infty$.
          By \cite[Lemma 3.1]{Keisler1964}, there exists some $a\in \overline{P_M}$ such that $[P_M(a):P_M]>n$.
    \item $S\ne \emptyset$:
          Take some $s_0\in S$ and define $F=P_M(S\setminus\set{s_0})$.
          Because $M$ is algebraically closed, there exists an $n+1$-th root $a=s_0^{\frac{1}{n+1}}\in M$.
          We know that $s_0$ is transcendental over $F$, so $a$ is also transcendental over $F$.
          \cref{transcendental degree} implies that $[F(a):F(s_0)]=n+1$,
          where $F(s_0)=P_M(S)$ and $F(a)=P_M(S,a)$, as needed.
  \end{enumerate}
\end{proof}

\begin{lemma} \label{extend to automorphism}
  Suppose $\trdeg(M/P_M)\ge\kappa$ (in particular, if $M$ is $\kappa$-saturated) and let $A,A'\subseteq M$ be subsets with  $\abs{A},\abs{A'}<\kappa$.
  If $f:P_M(A)\to P_M(A')$ is an isomorphism of fields that restricts to an $L$-automorphism $f|_{P_M}$, then $f$ can be extended to an automorphism of $M$.
\end{lemma}

\begin{proof}
    From transitivity of transcendental degree
    $$\trdeg(M/P_M)=\trdeg(M/P_M(A))+\trdeg(P_M(A)/P_M),$$
    and $\trdeg(P_M(A)/P_M)\le\abs{A}<\kappa$, so $\trdeg(M/P_M(A))=\trdeg(M/P_M)$.
    Similarly, $\trdeg(M/P_M(A'))=\trdeg(M/P_M)$.
    Let $S,S'\subseteq M$ be transcendence basis of $M$ over $P_M(A),P_M(A')$ respectively, $\abs{S}=\trdeg(M/P_M)=\abs{S'}$.
    Extend $f$ to an automorphism of fields $\sigma:M\to M$, by mapping $S\mapsto S'$ and extending to the algebraic closure arbitrarily.
    The restriction $\sigma|_{P_M}=f|_{P_M}$ is an $L$-automorphism of $P$, so \cref{isomorphism} implies that $\sigma$ is an $L^P$-automorphism.
\end{proof}

\section{Quantifier elimination and more}\label{subfield quantifier elimination and more}

\subsection{Completions}
Keisler \cite{Keisler1964} proved that $\ACF_T$ is complete when $T$ is a complete theory in the language of rings.
We generalize this by allowing the language of $T$ to be arbitrary.

In his proof, Keisler used special models.
We will instead use saturated models, which simplifies the proof, but requires an additional set-theoretic assumption (namely, the generalized continuum hypothesis). There are standard techniques from set theory that ensures the generalized continuum hypothesis from some point on while fixing a fragment of the universe (so this does not affect questions of e.g., completeness of a given theory), see \cite{halevi2021saturated}, and we will use this freely.

\begin{proposition} \label{complete}
    If $T$ is a complete theory of fields, then $\ACF_T$ is complete.    
\end{proposition}

\begin{proof}
    It is enough to show that if $M,N\vDash \ACF_T$ are saturated models of the same cardinality $\kappa$, then they are isomorphic (see the discussion above the proposition).
    By \cref{saturation P}, $P_M,P_N\vDash T$ are $\kappa$-saturated, and in particular $\abs{P_M}=\abs{P_N}=\kappa$.
    Because $T$ is complete, \cite[Theorem 5.1.13]{chang1990model} implies that there is an $L$-isomorphism $\sigma_0:P_M\to P_N$.
    By \cref{saturation trdeg}, $\trdeg(M/P_M)=\trdeg(N/P_N)=\kappa$.
    Let $S\subseteq M$, $S'\subseteq N$ be transcendence basis over $P_M,P_N$ respectively, $\abs{S}=\abs{S'}=\kappa$.
    We can extend $\sigma_0$ to an isomorphism of fields $\sigma_1:M\to N$, by mapping $S\mapsto S'$ and extending to the algebraic closure arbitrarily.
    The restriction $\sigma_1|_{P_M}$ is an $L$-isomorphism, so by \cref{isomorphism} $\sigma_1$ is an $L^P$-isomorphism.
\end{proof}

\subsection{Quantifier elimination}
Our proof of quantifier elimination will be essentially the same as Delon's \cite[Proposition 14]{Delon2012}.
One difference is that the criterion used by Delon to prove quantifier elimination assumes a countable language, so we will need a slightly generalized criterion.

 In \cite{HKR18}, Hils, Kamensky and Rideau proved the same result in a similar fashion.
 Our proof was derived independently, as we were not aware of their work during the research. 

We will need the following fact, which follows from \cite[Theorem 8.4.1]{hodges1993model}.
\begin{fact} \label{criterion for qe}
  A theory $T$ has quantifier elimination iff for any two models $M,N\vDash T$ such that $N$ is $\abs{M}^+$-saturated and any substructures $A\subseteq M$ and $A'\subseteq N$ with an isomorphism $\sigma:A\to A'$, $\sigma$ can be extended to an embedding $M\to N$.
\end{fact}

\begin{theorem} \label{quantifier elimination}
  If $T$ has quantifier elimination, then $\ACF_T^f$ has quantifier elimination.
\end{theorem}

\begin{proof}
  Let $M,N\vDash \ACF_T^f$ such that $N$ is $\abs{M}^+$-saturated.
  Let $A\subseteq M$, $A'\subseteq N$ be $L^f$-substructures with isomorphism $\sigma:A\to A'$.
  By \cref{field of fractions substructure}, $\fr(A)\subseteq M$, $\fr(A')\subseteq N$ are $L^f$-substructures with $P_{\fr(A)}=P_A$, $P_{\fr(A')}=P_{A'}$.
  We can extend $\sigma$ to an isomorphism of fields $\fr(A)\to \fr(A')$ that will have the same restriction $P_A\to P_{A'}$, and so by \cref{isomorphism} would still be an $L^f$-isomorphism.
  Thus, we can assume without loss of generality that $A$ and $A'$ are subfields.
  By \ref{saturation P}, $P_N$ is $\abs{M}^+$-saturated, and in particular $\abs{P_M}^+$-saturated.
  The restriction $\sigma|_{P_A}:P_A\to P_{A'}$ is an isomorphism of $L$-structures from \cref{isomorphism}, so quantifier elimination and \cref{criterion for qe} imply that we can extend $\sigma|_{P_A}$ to an embedding $\sigma_0:P_M\to P_N$.

  Let $B=\sigma_0(P_M)\subseteq P_N$.
  By \cref{substructure}, $A\forkindep^l_{P_A} P_M$ and $A'\forkindep^l_{P_{A'}} P_N$, in particular by monotonicity $A'\forkindep^l_{P_{A'}} B$.
  The field isomorphisms $\sigma:A\to A'$ and $\sigma_0:P_M\to B$ both restrict to the same isomorphism $P_A\to P_{A'}$, so there is a unique field isomorphism $\sigma_1:A.P_M\to A'.B$ such that $\sigma_1|_A=\sigma$, $\sigma_1|_{P_M}=\sigma_0$, by \cref{isomorphism amalgamation}.

  Let $S\subseteq M$ be a transcendental basis of $M$ over $A.P_M$, $\abs{S}\le\abs{M}$.
  From \cref{saturation trdeg} $\trdeg(N/P_N)\ge \abs{M}^+$ and $\abs{A'}=\abs{A}\le \abs{M}$, so there exists $S'\subseteq N$ algebraically independent over $A'.P_N$ with $\abs{S}=\abs{S'}$.
  Let $M'=\overline{A'.B(S')}\subseteq N$.
  Quantifier elimination implies that the substructure $B\subseteq P_N$ is elementary, so by \cref{elementary extension is regular} $B\subseteq P_N$ is regular.
  We also know that $A'\forkindep^l_{P_{A'}} P_N$, so by base monotonicity $A'.B\forkindep^l_{B} P_N$ and by \cref{regular algebraically independent} $\overline{A'.B(S')}\forkindep^l_B P_N$, where $\overline{A'.B(S')}=M'$.
  Thus, $M'\subseteq N$ is a substructure, with $P_{M'}=B$, from \cref{substructure}.

  We also have $M=\overline{A.P_M(S)}$, so we can extend $\sigma_1:AP_M\to A'B$ to $\sigma_2:M\to M'$ by mapping $S\mapsto S'$ arbitrarily and extending to the algebraic closure.
  In particular, $\sigma_2(P_M)=B=P_{M'}$ and $\sigma_2|_{P_M}=\sigma_0$ is an isomorphism of $L$-structures, so $\sigma_2$ is an isomorphism of $L^{f}$-structures by \cref{isomorphism}.
  Thus, $\sigma_2$ is an embedding of $M$ into $N$ that extends $\sigma$.
\end{proof}

\begin{corollary}[{\cite[Therorem 1]{Delon2012}}]
    $\ACF_\ACF^f$ eliminates quantifiers.
\end{corollary}

\begin{corollary}
    $\ACF_{\mathrm{RCF}}^f$ eliminates quantifiers, where RCF is the theory of real closed fields in the language $L_{\text{rings}}\cup\set{\le}$.
\end{corollary}

\begin{corollary}
    Let ACVF be the theory of algebraically closed valued fields in the divisibility language, that is the language of rings with a binary relation $x|y$ signifying $v(x)<v(y)$.
    ACVF eliminates quantifiers, so $\ACF_\mathrm{ACVF}^f$ eliminates quantifiers (by \cref{acvf nip} it is also NIP).
\end{corollary}

From quantifier elimination, we can deduce a couple of important corollaries.
Both corollaries will rely on expanding a theory $T$ to the Morleyzation, which has quantifier elimination, as defined below.

\begin{definition}
  For a theory $T$, the \emph{Morleyzation} $T_\mathrm{Mor}$ of $T$ is an expansion of $T$ by relations $R_\mathbb{\psi}(x)$ for any $\psi(x)\in L$, such that $T_{\mathrm{Mor}}\vdash\forall x(R_\psi(x)\leftrightarrow \psi(x))$.
  
\end{definition}

\begin{corollary} \label{bounded}
  Every formula $\phi(x)\in L^P$ is equivalent modulo $\ACF_T$ to a bounded formula, that is a formula where every quantifier is over $P$ (see \cref{def bounded}).
\end{corollary}

\begin{proof}
    Consider the Morleyzation $T_{\mathrm{Mor}}$ and the theory $\ACF_{T_\mathrm{Mor}}^f$ which has quantifier elimination by \cref{quantifier elimination}.
    In particular, $\phi(x)$ is equivalent to a quantifier free formula $\phi_0(x)\in L_\mathrm{Mor}^f$ modulo $\ACF_{T_\mathrm{Mor}}^f$.
    Replace all occurrences of $l_n$, $f_{n,i}$ in $\phi_0(x)$ with the formulas defining them, to get an equivalent formula $\phi_1(x)\in L_\mathrm{Mor}^P$.
    The formulas defining $l_n$, $f_{n;i}$ are bounded, so $\phi_1(x)$ is bounded.
    
    For any formula $\psi(y)\in L$ consider the bounded formula $\psi^P(y)\in L^P$ created from \cref{corresponding formula}.
    The axioms of $\ACF_{T_\mathrm{Mor}}$ (\cref{delons language}) imply that $\ACF_{T_\mathrm{Mor}}\vdash \forall y (R_{\psi}(y)\leftrightarrow \psi^P(y))$.
    Replace each predicate $R_\psi(y)$ in $\phi_1(x)$ by the corresponding $\psi^P(y)$, to get a bounded formula $\phi_2(x)\in L^P$ which is equivalent to $\phi(x)$ modulo $\ACF_T$.
\end{proof}

\begin{remarkcnt}
    In that case that $L$ is the language of rings, \cref{bounded} follows from \cite[Proposition 2.1]{Casanovas2001stable}, because ACF has \emph{nfcp} and $P_M$ is small in any model $M\vDash \ACF_T$ (as witnessed in a saturated extension, by \cref{saturation trdeg}).
\end{remarkcnt}

\begin{corollary} \label{elementary map}
  Let $M,N\vDash \ACF_T^f$ and let $A\subseteq M$, $B\subseteq N$ be substructures.
  Then $\sigma:A\to B$ is a partial elementary map from $M$ to $N$ iff $\sigma: A\to B$ is an isomorphism of rings such that $\sigma(P_A)=P_B$ and $\sigma|_{P_A}:P_A\to P_B$ is a partial elementary map from $P_M$ to $P_N$.
\end{corollary}

\begin{proof}
  Suppose $\sigma:A\to B$ is a partial elementary map from $M$ to $N$ in $\ACF_T^f$.
  Then $\sigma$ is in particular an isomorphism, so $\sigma(P_A)=P_B$.
  The restriction $\sigma|_{P_A}$ is a partial elementary map from $P_M$ to $P_N$ in $T$, because for every formula $\phi(x)\in T$, we can apply \cref{corresponding formula} to get  $\phi^P(\bar{x})\in \ACF_T$, such that
  $\phi(P_B)=\phi^P(B)=\sigma(\phi^P(A))=\sigma(\phi(P_A))$.

  For the other direction, suppose $\sigma: A\to B$ is an isomorphism of rings such that $\sigma(P_A)=P_B$ and $\sigma|_{P_A}:P_A\to P_B$ is a partial elementary map from $P_M$ to $P_N$ in $T$.
  In particular, $P_M$ and $P_N$ have the same theory, so we can assume that $T$ is the complete theory $T=\mathrm{Th}(P_M)=\mathrm{Th}(P_N)$.
  Let $T_{\mathrm{Mor}}$ be the Morleyzation of $T$, $T_{\mathrm{Mor}}$ has quantifier elimination.
  We can expand the language of $P_M$ and $P_N$ by definable relations to get $P_M,P_N\vDash T_{\mathrm{Mor}}$.
  With this expanded language $M,N\vDash \ACF_{T_{\mathrm{Mor}}}^f$.
  The expansion is only relational, so we can still consider $A$ and $B$ as substructure.
  The restriction $\sigma|_{P_A}$ is a partial elementary map in $T$, so it is an isomorphism in $T_{\mathrm{Mor}}$, and thus by \cref{isomorphism} $\sigma$ is an isomorphism in $\ACF_{T_{\mathrm{Mor}}}^f$.
  By \cref{complete,quantifier elimination} $\ACF_{T_{\mathrm{Mor}}}^f$ is complete and eliminates quantifiers, so $\sigma$ is a partial elementary map in $\ACF_{T_{\mathrm{Mor}}}^f$.
  In particular, it is a partial elementary map in $\ACF_T^f$.
\end{proof}

Using this result on elementary maps, we can now show that \cref{saturation P,saturation trdeg} fully characterize the saturated models of $\ACF_T$.

\begin{proposition} \label{saturated}
  Suppose $\kappa>\abs{L}$, then $N\vDash \ACF_T$ is $\kappa$-saturated iff $P_N\vDash T$ is $\kappa$-saturated and $\trdeg(N/P_N)\ge\kappa$
\end{proposition}

\begin{proof}
  The first direction, if $N\vDash \ACF_T$ is $\kappa$-saturated, then $P_N\vDash T$ is $\kappa$-saturated and $\trdeg(N/P_N)\ge\kappa$, is proved in \cref{saturation P,saturation trdeg}.
  For the other direction, we will prove $\kappa$-homogeneity and $\kappa^+$-universality.
  By expanding the language with definable relations and functions, we can assume $N\vDash \ACF_T^f$.
  Let $A,B\subseteq N$ and let $\sigma:A\to B$ be a partial elementary map in $N$ with $\sigma(A)=B$, such that $\abs{A}=\abs{B}<\kappa$.
  Without loss of generality, we can assume that $A,B\subseteq N$ are $L^f$-substructures, and by \cref{field of fractions substructure} we can also assume they are subfields.
  \cref{elementary map} implies that $\sigma|_{P_A}:P_A\to P_B$ is a partial elementary map in $P_N$.
  We know that $P_N$ is $\kappa$-homogeneous and $\abs{P_A}=\abs{P_B}<\kappa$, so we can extend $\sigma|_{P_A}$ to an automorphism $\sigma_0:P_N\to P_N$ in $T$.

   We have $A\forkindep^l_{P_A} P_N$ and $B\forkindep^l_{P_B} P_N$ from \cref{substructure}, and the field isomorphisms $\sigma$ and $\sigma_0$ restrict to the same isomorphism $P_A\to P_B$, so by \cref{isomorphism amalgamation} they can be jointly extended to an isomorphism of fields $\sigma_1:A.P_N\to B.P_N$.
  From \cref{extend to automorphism}, $\sigma_1$ can be extended to an automorphism of fields $\sigma_2:N\to N$.
  \cref{isomorphism} implies that $\sigma_2$ is an $L^f$ automorphism because $\sigma_2|_{P_N}=\sigma_0$ is an automorphism in $T$, and $\sigma_2$ extends $\sigma$ as needed.

  Now Let $M\models\ACF_T$ with $\abs{M}\le\kappa$, by expanding the language we can assume $M\vDash \ACF_T^f$.
  We have $P_M\vDash T$ with $\abs{P_M}<\kappa$, so by $\kappa^+$-universality of $P_N$ there exists an elementary embedding $\tau_0:P_M\to P_N$.
  Let $B=\tau_0(P_M)$.
  We have $B\preceq P_N$, and in particular from \cref{elementary extension is regular} $B\subseteq P_N$ is a regular extension.
  Let $S$ be a transcendental basis of $M$ over $P_M$, $\abs{S}\le \kappa$ and $\trdeg(N/P_N)\ge \kappa$, so there exists $S_0\subseteq N$ algebraically independent over $P_N$ with $\abs{S_0}=\abs{S}$.
  We can extend $\tau_0$ to an embedding $\tau_1:M\to N$ by mapping $S\mapsto S_0$ arbitrarily and extending to the algebraic closure.
  Let $M_0=\tau_1(M)=\overline{B(S_0)}$.
  From \cref{regular algebraically independent}, $\overline{B(S_0)}\forkindep^l_B P_N$, so by \cref{submodel} $M_0\subseteq N$ is an $L^f$-substructure with $P_{M_0}=B$.
  We have that $\tau_1:M\to M_0$ is an isomorphism of fields with $\tau_1|_{P_M}=\tau_0:P_M\to P_{M_0}$ an elementary embedding, so by \cref{elementary map} $\tau_1$ is an elementary embedding.
\end{proof}

\subsection{Model completeness}
In \cite[Corollary 15]{Delon2012}, Delon proved that $\ACF_\ACF^{ld}$ is model complete.
We can show that if $T$ is model complete, then $\ACF_T^{ld}$ is model complete, but in fact we only need a weaker condition --- that regular extensions in $T$ are elementary.

\begin{theorem} \label{model completeness}
  The following are equivalent:
  \begin{enumerate}
    \item $\ACF_T^f$ is model complete.
    \item $\ACF_T^{ld}$ is model complete.
    \item For any $Q,R\vDash T$ such that $Q\subseteq R$ is a substructure, if $Q\subseteq R$ is a regular extension, then $Q\preceq R$.
    \item $T_{\mathrm{reg}}$ (\cref{language of regular extensions}) is model complete.
  \end{enumerate}
\end{theorem}

\begin{proof}
  $1\implies 2$:
  Let $M,N\vDash \ACF_T^{ld}$ with $M\subseteq N$ an $L^{ld}$-substructure.
  We can expand $M$ and $N$ uniquely to models of $\ACF_T^f$, by \cref{submodel} $M\subseteq N$ is an $L^f$-substructure.
  $\ACF_T^f$ is model complete, so $M\preceq N$ in $L^f$, in particular $M\preceq N$ in $L^{ld}$.

  $2\implies 3$:
  Let $Q,R\vDash T$ with $Q\subseteq R$ a regular extension.
  We will construct $M,N\vDash \ACF_T^{ld}$ such that $P_M=Q$, $P_N=R$ and $M\subseteq N$.
  We would have liked to take $M=\overline{Q}$, but then we may have $[M:Q]<\infty$, so we should make $M$ a bit larger.
  Let $s$ be a new element, transcendental over $R$.
  The subfield $Q\subseteq R$ is regular, so by \cref{regular algebraically independent} $\overline{Q(s)}\forkindep^l_Q R$.
  Define $M=\overline{Q(s)}$, $Q\subseteq M$ is not an algebraic extension so in particular $[M:Q]=\infty$.
  We have $M\vDash \ACF_T^{ld}$, where we define $P_M=Q$.
  Similarly, define $N=\overline{R(s)}$, $N\vDash \ACF_T^{ld}$ with $P_N=R$.
  We know that $P_M\subseteq P_N$ is an $L$-substructure and $M\forkindep^l_{P_M} P_N$, so by \cref{submodel} $M\subseteq N$ is an $L^{ld}$-substructure.
  Model completeness implies $M\preceq N$, and in particular $P_M\preceq P_N$, because for every formula $\phi(\bar{x})\in L$ we have $P_M\vDash \phi(\bar{a})\iff M\vDash \phi^P(\bar{a})\iff N\models\phi^P(\bar{a})\iff P_N\vDash \phi(\bar{a})$ for every $\bar{a}\in P_M$, where $\phi^P$ is given by \cref{corresponding formula}.
  
  $3\implies 4$:
  Let $Q,R\vDash T_\mathrm{reg}$ be such that $Q\subseteq R$ is an $L_\mathrm{reg}$-extension.
  By \cref{substructure iff regular}, $Q\subseteq R$ is a regular field extension, so $Q\preceq R$ in $L$ by assumption.
  Because $L_\mathrm{reg}$ is an expansion by definable relations, $Q\preceq R$ also in $L_\mathrm{reg}$.

  $4\implies 1$:
  Let $M,N\vDash \ACF_T^f$ and suppose $M\subseteq N$ is a substructure.
  \cref{submodel} implies that $P_M\subseteq P_N$ is an $L$-substructure and $M\forkindep^l_{P_M} P_N$.
  However, $M$ is algebraically closed, so by monotonicity $\overline{P_M}\forkindep^l_{P_M} P_N$, that is $P_M\subseteq P_N$ is a regular extension.
  Extending $P_M$ and $P_N$ to models $T_\mathrm{reg}$, we see by \cref{substructure iff regular} that $P_M\subseteq P_N$ is an $L_\mathrm{reg}$-extension, so $P_M\preceq P_N$ by assumption.
  The inclusion map $M\to N$ restricts to the elementary inclusion $P_M\to P_N$, so by \cref{elementary map}, $M\preceq N$.
\end{proof}

\begin{corollary}[{\cite[Corollary 15]{Delon2012}}] $\ACF_\ACF^{ld}$ is model complete.
\end{corollary}

\begin{corollary}
    $\ACF_{\mathrm{PSF}}^{ld}$ is model complete, where PSF is the theory of pseudo-finite fields in the language of rings (see \cref{psf model complete} for a proof).
\end{corollary}

\begin{remarkcnt}
    $\ACF_\ACF$ is \emph{not} model complete.
    By \cite[page 207]{tent_ziegler_2012}, the pregeometry of  an algebraically closed field $K$ of transcendence degree at least 4 over its prime field with algebraic independence is not modular: there are algebraically closed subfields $A,B\subseteq K$ such that $A\not\forkindep_{A\cap B}^\ACF B$.
    Define
    \begin{align*}
        M&=A & N&=K\\
        P_M&=A\cap B & P_N&=B.
    \end{align*}
    It is clear that $M\subseteq N$ is an $L^P$-substructure, however if $M\preceq N$, then \cref{submodel} would imply that $A\forkindep^l_{A\cap B} B$, and in particular $A\forkindep^\ACF_{A\cap B} B$, a contradiction.
\end{remarkcnt}

\section{Classification and independence}\label{subfield classification and independence}
In this section we will assume that $T$ is complete (\cref{complete} implies that $\ACF_T$ is also complete) and we will work inside a monster model $\mathbb{M}\vDash \ACF_T$.
Denote $P:=P_\M$.

Assuming $T$ is $\NSOP_1$, we will define an independence relation $\forkindep^*$ on $\M$ and prove that it implies Kim-dividing (in fact, Kim$^u$ dividing, see \cref{def kim u dividing}) 
With this result, we will prove that $\ACF_T$ is $\NSOP_1$ and that under certain conditions $\forkindep^*$ is the Kim-independence. 
We will then expand this result to simplicity and stability.

We will also prove that stability lifts from $T$ to $\ACF_T$ using a different approach, by counting types.
This approach will let us extend the result to $\lambda$-stability.

Finally, we will prove that NIP lifts from $T$ to $\ACF_T$,

\subsection{Kim-dividing}

\begin{definition}
  Call a subfield $A\subseteq \M$ \emph{D-closed} (D for Delon's language) if it is closed under the functions $f_{n,i}$, or equivalently if $A\forkindep^l_{P_A} P$.
  For a set $B\subseteq \M$, denote by $\vect{B}_D$ the D-closure of $B$, that is the smallest field containing $B$ and closed under $f_{n,i}$.
\end{definition}

\begin{remarkcnt}\label{shorter local character}
We have the following remarks on D-closure:
\begin{itemize}
    \item In \cite[Definition 3.1]{MARTIN_PIZARRO_2020}, the condition D-closed was called $P$-special.
    \item If $A\subseteq \M$ is definably closed in $L^P$, then it is D-closed.
    In particular, for every $A\subseteq \M$, $\dcl(A)$ and $\acl(A)$ are D-closed.
    \item D-closure gives a shorter proof of local character of $\forkindep^l$ (see \cref{linear disjointedness properties}).
    Suppose $a$ is finite and $P$ is an infinite field.
    Let $A=\vect{a}_D$ be the D-closure of $a$ inside the pair of fields $(P(a), P)$.
    Consider $P_A=P \cap A$, which is countable.
    We have $P_A(a)\subseteq A$, so by monotonicity $P_A(a)\forkindep^l_{P_A} P$.
\end{itemize}
\end{remarkcnt}

\begin{lemma} \label{linear disjointedness AP}
  Suppose $A, B, C\subseteq \M$ are subfields with $C\subseteq A\cap B$. 
  If $A$ is D-closed, then $A.P\forkindep^l_{C.P} B.P$ iff $A\forkindep^l_{C.P_A} B.P$.
  By symmetry, if $B$ is D-closed, then $A.P\forkindep^l_{C.P} B.P$ iff $A.P\forkindep^l_{C.P_B} B$.
  Furthermore, if both $A$ and $B$ are D-closed, then $A.P\forkindep^l_{C.P} B.P$ implies $A.B\forkindep^l_{P_A.P_B} P$, i.e. $P_{A.B}=P_A.P_B$ and $A.B$ is D-closed.
\end{lemma}

\begin{proof}
  If $A\forkindep^l_{C.P_A} B.P$, then $A.P\forkindep^l_{C.P} B.P$ from base monotonicity.
  On the other hand, if $A.P\forkindep^l_{C.P} B.P$, then because $A\forkindep^l_{P_A} P$ implies $A\forkindep^l_{C.P_A} C.P$ from base monotonicity, we get from transitivity that $A\forkindep^l_{C.P_A} B.P$.
  For the furthermore part, we know from $A\forkindep^l_{P_A} P$ and $A.P\forkindep^l_{C.P} B.P$ that $A\forkindep^l_{C.P_A} B.P$.
  By base monotonicity, $A.B\forkindep^l_{B.P_A} B.P$.
  Also, from $B\forkindep^l_{P_B} P$ and base monotonicity, $B.P_A\forkindep^l_{P_A.P_B} P$, thus by transitivity $A.B\forkindep^l_{P_A.P_B} P$. 
\end{proof}

\begin{definition} \label{D:star independence}
  Let $M\preceq \M$ and $A,B\subseteq \M$ be small D-closed subfields, such that $M\subseteq A\cap B$.
  Define $A\forkindep^*_M B$ if
  \begin{enumerate}
    \item $P_A\forkindep^K_{P_M} P_B$ in $P$.
    \item $A.P\forkindep^l_{M.P} B.P$.
  \end{enumerate}

\end{definition}

\begin{lemma} \label{coheir independence implies}
  Let $A,B,C\subseteq \mathbb{M}$ be small subsets with $C\subseteq A\cap B$.
  If $A\forkindep^u_C B$, then:
  \begin{enumerate}
    \item $P_A\forkindep^u_{P_C} P_B$ in $P$.
    \item If $A$, $B$ and $C$ are subfields and $B$ is D-closed, then $A.P\forkindep^l_{C.P} B.P$.
  \end{enumerate}
  In particular, if $M\preceq \M$ and $A$ and $B$ are D-closed with $M\subseteq A\cap B$, then $A\forkindep^u_M B$ implies $A\forkindep^*_M B$.
\end{lemma}

\begin{proof}
  For point (1), suppose $P\vDash \phi(a, b)$ for some formula $\phi(x,y)\in L$, $a\in P_A$ and $b\in P_B$.
  Let $\phi^P(x,y)\in L^P$ be as in \cref{corresponding formula}, we have $\mathbb{M}\vDash \phi^P(a,b)$.
  By $A\forkindep^u_C B$ there is some $c\in C$ such that $\mathbb{M}\vDash \phi^P(c,b)$.
  Thus, $c\in P\cap C=P_C$, and we have $P\vDash \phi(c, b)$.

  For point (2), by \cref{linear disjointedness AP} it is enough to prove $A.P\forkindep^l_{C.P_B} B$ .
  Let $\sum_i u_ib_i=0$ for $u_i\in A.P$ and $b_i\in B$ such that the $u_i$ are not all equal to $0$.
  We can write $u_i=f_i(\bar{a}_i,\bar{p}_i)$ for $f_i\in C(\bar{x}_i,\bar{y}_i)$ rational functions, $\bar{a}_i\in A$ and $\bar{p}_i\in P$.
  Assume that $f_i$ are polynomials by multiplying by all denominators.
  We have 
  $$\vDash \sum_i f_i(\bar{a}_i,\bar{p}_i)b_i=0\land \bigvee_if_i(\bar{a}_i,\bar{p}_i)\ne 0,$$
  and in particular 
  $$\vDash \exists\bar{y}_i\in P, \sum_i f_i(\bar{a}_i,\bar{y}_i)b_i=0\land \bigvee_if_i(\bar{a}_i,\bar{y}_i)\ne 0.$$
  From $A\forkindep^u_C B$, there are $\bar{c}_i\in C$ such that 
  $$\vDash \exists\bar{y}_i\in P \sum_i f_i(\bar{c}_i,\bar{y}_i)b_i=0\land \bigvee_if_i(\bar{c}_i,\bar{y}_i)\ne 0.$$
  Let $\bar{q}_i\in P$ witness the existence, and let $v_i=f_i(\bar{c}_i,\bar{q}_i)\in C.P$.
  We have $\sum_i v_ib_i=0$ and $v_i$ are not all equal to 0.
  Moreover,  $B\forkindep^l_{P_B} P$, so by base monotonicity $B\forkindep^l_{C.P_B} C.P$, thus there are $w_i\in C.P_B$, not all equal to 0, such that $\sum_i w_ib_i=0$, as needed.

  The ``in particular'' part follows from the definition of $\forkindep^*$, because $P_A\forkindep^u_{P_M} P_B$ implies $P_A\forkindep^K_{P_M} P_B$ (see \cite[Fact 3.10]{D_ELB_E_2021_acfg}).
\end{proof}

\begin{lemma} \label{coheir sequence in P}
  Let $A,B,C\subseteq \M$ be small subsets with $C\subseteq A\cap B$ and let $(B_i)_{i<\omega}$ be a $C$-indiscernible coheir sequence such that $B\equiv_A B_i$ in $\ACF_T$, then $(P_{B_i})_{i<\omega}$ is a $P_C$-indiscernible coheir sequence such that $P_B\equiv_{P_A} P_{B_i}$ in $P$.
\end{lemma}

\begin{proof}
  For every formula in $P$, we can restrict all quantifiers and free variables to be over $P$ to get a formula in $\M$ with the same definable set.
  This proves that $(P_{B_i})_{i<\omega}$ is $P_C$-indiscernible and $P_B\equiv_{P_A} P_{B_i}$ in $P$.
  From \cref{coheir independence implies}, $P_{B_i}\forkindep^u_{P_C} P_{B_{<i}}$ in $P$, and $P_{B_{<i}}=\bigcup_{j<i}P_{B_j}$, so $(P_{B_i})_{i<\omega}$ is a $P_C$-indiscernible coheir sequence.
\end{proof}

\begin{proposition} \label{star independence amalgamation}
  Assume $T$ is $\NSOP_1$.
  Let $M\preceq \mathbb{M}$ and let $A,B\subseteq \mathbb{M}$ be small D-closed subfields with $M\subseteq A\cap B$, such that $A$ is algebraically closed as a field.
  If $A\forkindep^*_M B$, then $\tp(A/B)$ does not Kim$^u$-divide over $M$ (recall \cref{def kim u dividing}).
\end{proposition}

\begin{proof}
  Let $(B_i)_{i<\omega}$ be any $M$-indiscernible coheir sequence such that $B\equiv_M B_i$ in $\ACF_T$ for every $i<\omega$ and let $\beta_i:B\to B_i$ be $L^P$-isomorphisms such that $(\beta_i(b))_{b\in B}$ is an $M$-indiscernible coheir sequence in $\ACF_T$.
  By \cref{coheir sequence in P}, $(P_{B_i})_{i<\omega}$ is a $P_M$-indiscernible coheir sequence in $P$, where $P_{B_i}$ is enumerated as $(\beta_i(b))_{b\in P_B}$.
  Because $T$ is $\NSOP_1$ and $P_A\forkindep^K_{P_M} P_B$ in $P$, \cref{kim independence in nsop1}(2) implies that there exists $Q\subseteq P$ such that $P_AP_B\equiv_{P_M} QP_{B_i}$ in $P$ for all $i<\omega$, where we consider all the above fields as tuples. More explicitly, let $p((x_a)_{a\in P_A},(x_b)_{b\in P_B})=\tp((a)_{a\in P_A},(b)_{b\in P_B}/P_M)$, then let $(a')_{a\in P_A}$ be a realization of $\bigcup_{i<\omega} p((x_a)_{a\in P_A},(\beta_i(b))_{b\in P_B})$, and let $Q$ be $\set{a' \mid a \in P_A}$.
  As $(a)_{a\in P_A}(b)_{b\in P_B} \equiv_{P_M} (a')_{a\in P_A}(\beta_i(b))_{b\in P_B}$ in $P$, by saturation there are automorphisms $\gamma_i$ of $P$ mapping $P_AP_B$ to $QP_{B_i}$ extending $\beta_i|_{P_B}$ (so fixing $P_M$ pointwise) such that $\gamma_i(a)=a'$ for all $a\in P_A$. 
  In particular, the restrictions $\gamma_i|_{P_A}:P_A\to Q$ are the same for every $i<\omega$.
  Name this restriction $\alpha_0:P_A\to Q$.

  Let $S\subseteq A$ be a transcendence basis of $A$ over $M.P_A$.
  \Cref{saturation trdeg} implies that $\trdeg(\mathbb{M}/P)=\abs{\mathbb{M}}$, so there exists some $S'$ algebraically independent over $B_{<\omega}P$ with $\abs{S'}=\abs{S}$.
  Define $A'=\overline{M.Q(S')}$.
  From \cref{substructure}, $M\forkindep^l_{P_M} P$, so from monotonicity $M\forkindep^l_{P_M} P_A$ and $M\forkindep^l_{P_M} Q$.
  Thus, from stationarity of $\forkindep^l$, we can extend $\alpha_0:P_A\to Q$ to an isomorphism of fields $M.P_A\to M.Q$ preserving $M$ pointwise.
  Map $S\mapsto S'$ arbitrarily and extend arbitrarily to the algebraic closure, to get an isomorphism of fields $\alpha:A\to A'$.
  This give us a way to consider $A'$ as a tuple.

  Let $i<\omega$.
  We know that $B\forkindep^l_{P_B} P$ and $B_i\forkindep^l_{P_{B_i}} P$, the field isomorphisms $\beta_i:B\to B_i$ and $\gamma_i:P\to P$ both restrict to the same isomorphism $P_B\to P_{B_i}$, so from \cref{isomorphism amalgamation} they can be jointly extended to an isomorphism of fields $\sigma_{i,0}:B.P\to B_i.P$.
  From $A.P\forkindep^l_{M.P} B.P$ and \cref{linear disjointedness AP} we get that $A\forkindep^l_{M.P_A} B.P$.
  We would like to prove that also $A'\forkindep^l_{M.Q} B_i.P$.
  We know that $A$ is algebraically closed, so $M.P_A\subseteq B.P$ is regular.
  Applying \cref{regular extension isomorphism} with $\sigma_{i,0}$, we get that $M.Q\subseteq B_i.P$ is regular.
  The set $S'$ is algebraically independent over $B_i.P$, so from \cref{regular algebraically independent} $\overline{M.Q(S')}\forkindep^l_{M.Q} B_i.P$, where $\overline{M.Q(S')}=A'$.

  The isomorphisms of fields $\alpha:A\to A'$ and $\sigma_{i,0}:B.P\to B_i.P$ restrict to the same isomorphism $M.P_A\to M.Q$, which acts as $\alpha_0$ on $P_A$ and preserves $M$ pointwise.
  Thus, from \cref{isomorphism amalgamation}, they can be jointly extended to an isomorphism of fields $\sigma_{i,1}:A.B.P\to A'.B_i.P$.
  By \cref{extend to automorphism}, $\sigma_{i,1}$ can be extended to $\sigma_{i,2}$ an $L^P$-automorphism of $\mathbb{M}$.
  The automorphism $\sigma_{i,2}$ maps $AB\mapsto A'B_i$ and extends $\alpha$ and $\beta_i$ (in particular fixes $M$ pointwise). Let $q((x_a)_{a\in A},(x_b)_{b\in B})=\tp((a)_{a\in A},(b)_{b\in B}/M)$. We get that $(\alpha(a))_{a\in A}$ realizes $\bigcup_{i<\omega} q((x_a)_{a\in A},(\beta_i(b))_{b\in B})$ as required.
\end{proof}

\subsection{NSOP$_1$, simplicity}
\begin{remarkcnt} \label{coheir acl closure}
    In a general theory $T$, if $A\forkindep^u_C B$, then $\acl(AC)\forkindep^u_{\acl(C)} \acl(BC)$.
    Indeed, by extension, for some $A'\equiv_{BC} A$ we have $A'\forkindep^u_C \acl(BC)$, and by applying an automorphism taking $A'$ to $A$ and fixing $BC$ we get that $A\forkindep^u_C \acl(BC)$.
    By base monotonicity, $A\forkindep^u_{\acl(C)} \acl(BC)$.
    
    Suppose that $\models\phi(d,b)$ where $\phi(x,y)$ is a formula over $\acl(C)$, $d\in \acl(AC)$ and $b\in \acl(BC)$.
    Let $\psi(x,z)$ be a formula over $C$ and $a\in A$ be such that $\psi(x,a)$ is algebraic, say of size $n$, and $\vDash \psi(d,a)$, that is
    $$\models\exists^{\le n}x\,\psi(x,a)\land \exists x\,(\phi(x,b)\land \psi(x,a)).$$
    As $A\forkindep^u_{\acl(C)} \acl(BC)$, there exists $c\in \acl(C)$ such that $\psi(x,c)$ is of size at most $n$ and $\vDash \exists x(\phi(x,b)\land \psi(x,c))$, let $e$ witness the existence.
    The fact that $\vDash \psi(e,c)$ implies that $e\in \acl(C)$, and we have $\vDash \phi(e,b)$, so $\acl(AC)\forkindep^u_{\acl(C)} \acl(BC)$.
\end{remarkcnt}

\begin{theorem} \label{nsop1}
  If $T$ is $\NSOP_1$, then $\ACF_T$ is $\NSOP_1$.
\end{theorem}

\begin{proof}
  We will use \cref{WIA}.
  Let $M\preceq \mathbb{M}$ and suppose $A_0$, $A_1$, $B_0$ and $B_1$ are such that $A_0B_0\equiv_M A_1B_1$ in $\ACF_T$, $B_1\forkindep^u_M B_0$ and $B_i\forkindep^u_M A_i$ for $i=0,1$.
  By \cref{coheir acl closure}, we can assume that $A_i=\acl(A_iM)$, $B_i=\acl(B_iM)$, and in particular they are all D-closed and algebraically closed.
  
  From $B_0\forkindep^u_M A_0$, we get using \cref{coheir independence implies} that $B_0\forkindep^*_M A_0$.
  However, $T$ is $\NSOP_1$, so \cref{kim independence in nsop1}(3) implies that $\forkindep^K$ in $P$ is symmetric, thus $\forkindep^*$ is also symmetric and we have $A_0\forkindep^*_M B_0$.
  By \cref{star independence amalgamation}, $\tp(A_0/B_0)$ does not Kim$^u$-divide over $M$.
  Extend the pair $(B_0,B_1)$ to a coheir sequence $(B_i)_{i<\omega}$ (to do that, first extend $\tp(B_1/MB_0)$ to a global type which is finitely satisfiable in $M$, and then generate a Morley sequence in that type; see \cite[\S 3.1]{Kaplan2017}). By the definition of Kim$^u$-dividing (\cref{def kim u dividing}) we get that there exists $A\subseteq \mathbb{M}$ such that $A_0B_0\equiv_M AB_0\equiv_M AB_1$ in $\ACF_T$.
\end{proof}

\begin{corollary}
    The theory of $\omega$-free PAC fields was shown to be non-simple by Chatzidakis \cite{Chatzidakis1999}, as it is PAC and unbounded, and NSOP$_1$ by Chernikov and Ramsey \cite{Chernikov2016}.
    Thus, $ACF_{\omega\text{-free PAC}}$ is NSOP$_1$ and non-simple as the theory of $\omega$-free PAC fields is interpretable in $ACF_{\omega\text{-free PAC}}$.
\end{corollary}

Now we will show that in NSOP$_1$ theories, Kim-independence is $\forkindep^*$ for certain sets.

\begin{proposition} \label{kim independence}
  Assume $T$ is $\NSOP_1$.
  Let $M\preceq \mathbb{M}$ and let $A,B\subseteq \mathbb{M}$ be small D-closed subfields with $M\subseteq A\cap B$.
  Then $A\forkindep^K_M B$ implies $A\forkindep^*_M B$.
  If either $A$ or $B$ are algebraically closed as fields, then also $A\forkindep^*_M B$ implies $A\forkindep^K_M B$.
\end{proposition}

\begin{proof}
  We will first prove that $A\forkindep^K_M B$ implies $A\forkindep^*_M B$.
  Suppose $A\forkindep^K_M B$, we need to prove that $P_A\forkindep^K_{P_M} P_B$ in $P$ and $A.P\forkindep^l_{M.P} B.P$.
  Take an arbitrary $M$-indiscernible coheir sequence $(B_i)_{i<\omega}$, with $B\equiv_M B_i$ in $\ACF_T$.
  The theory $T$ is $\NSOP_1$, so $\ACF_T$ is also $\NSOP_1$ from \cref{nsop1}.
  By \cref{kimu dividing,kim independence in nsop1}(2) there exists $A'\subseteq \M$ such that $AB\equiv_M A'B_i$ in $\ACF_T$.
  In particular, by $A\equiv_M A'$ in $\ACF_T$ there exists an automorphsim $\sigma$ of $\M$ mapping $A'$ to $A$ and preserving $M$ pointwise.
  Letting $B_i'=\sigma(B_i)$, $(B_i')_{i<\omega}$ is an $M$-indiscernible coheir sequence with $B\equiv_A B_i'$ in $\ACF_T$.
  By \cref{coheir sequence in P}, $(P_{B_i'})_{i<\omega}$ is a $P_M$-indiscernible coheir sequence with $P_B\equiv_{P_A} P_{B_i'}$ in $P$. 
  Because $T$ is $\NSOP_1$, \cref{kim independence in nsop1}(1) implies that $P_A\forkindep^K_{P_M} P_B$ in $P$.
  
  To prove that $A.P\forkindep^l_{M.P} B.P$, it is enough to prove that $A\forkindep^l_{M.P_A} B.P$, by \cref{linear disjointedness AP}.
  Let $\overline{a}\in A$ be a finite tuple and suppose it is linearly dependent over $B.P$.
  Because $A\forkindep^K_M B$, we can construct an uncountable $M$-indiscernible coheir sequence $(B_i)_{i<\omega_1}$, with $B\equiv_A B_i$ in $\ACF_T$.
  Let $\sigma_i\in \mathrm{Aut}(\mathbb{M}/A)$ be an automorphism mapping $B$ to $B_i$.
  We know that $\sigma_i$ preserves $P$ setwise, so by applying $\sigma_i$ we get that $\overline{a}$ is linearly dependent over $B_i.P$.
  By local character, there is some countable subfield $C\subseteq \acl(B_{<\omega_1}).P$ such that $C(\overline{a})\forkindep^l_C \acl(B_{<\omega_1}).P$.
  Because $C$ is countable, there is some $i<\omega_1$ such that $C\subseteq \acl(B_{<i}).P$.
  By \cref{coheir acl closure} we have $B_i\forkindep^u_M \acl(B_{<i})$, so \cref{coheir independence implies} implies that $B_i.P\forkindep^l_{M.P} \acl(B_{<i}).P$, and in particular from monotonicity $B_i.P\forkindep^l_{M.P} M.P.C$.
  However, the fact that $C(\overline{a})\forkindep^l_C \acl(B_{<\omega_1}).P$ also implies, using monotonicity, base monotonicity and symmetry, that $B_i.P.C\forkindep^l_{M.P.C} M.P.C(\overline{a})$, so by transitivity $B_i.P\forkindep^l_{M.P} M.P.C(\overline{a})$.
  The tuple $\overline{a}$ is linearly dependent over $B_i.P$, so it is linearly dependent over $M.P$.
  However, $A$ is D-closed so $A\forkindep^l_{P_A} P$ and by base monotonicity $A\forkindep^l_{M.P_A} M.P$.
  Thus, $\overline{a}$ is linearly dependent over $M.P_A$, as needed.

  If $A$ is algebraically closed and $A\forkindep^*_M B$, then from \cref{star independence amalgamation} $\mathrm{tp}(A/B)$ does not Kim$^u$-divide over $M$.
  $\ACF_T$ is $\NSOP_1$, so by \cref{kimu equals kim} Kim$^u$-dividing is the same as Kim-dividing, and by \cref{kim independence in nsop1}(2) Kim-dividing is the same as Kim-forking, thus $A\forkindep^K_M B$.
  The case where $B$ is algebraically closed follows from symmetry of $\forkindep^*$ and $\forkindep^K$ (\cref{kim independence in nsop1}(3)).
\end{proof}

\begin{remarkcnt}
    The proof of \cref{kim independence} was inspired by the proof of \cite[Proposition 7.3]{BENYAACOV2003235}
\end{remarkcnt}

\begin{theorem} \label{simple}
  If $T$ is simple, then $\ACF_T$ is simple. 
\end{theorem}

\begin{proof}
  Suppose $T$ is simple, in particular $T$ is $\NSOP_1$ so \cref{nsop1} implies that $\ACF_T$ is $\NSOP_1$.
  By \cref{kim independence in nsop1}(5), for an $\NSOP_1$ theory being simple is equivalent to Kim-independence having base monotonicity. 
  Let $A,B\subseteq \M$ be small subsets and $M,N\preceq \M$ submodels, such that $M\subseteq A$, $M\subseteq N\subseteq B$.
  Suppose $A\forkindep^K_M B$, we want to prove $A\forkindep^K_N B$.
  Without loss of generality we can assume that $A$ and $B$ are $\acl$-closed.

  By \cref{kim independence}, $A\forkindep^K_M B$ implies $A\forkindep^*_M B$.
  We have $A.P\forkindep^l_{M.P} B.P$, and by monotonicity $A.P\forkindep^l_{M.P} N.P$, so from \cref{linear disjointedness AP} $N.A$ is D-closed.
  Since $B$ is D-closed and algebraically closed as a field, by \cref{kim independence} it is enough to prove $N.A\forkindep^*_N B$.
  By base monotonicity of linear disjointness, $A.P\forkindep^l_{M.P} B.P$ implies $N.A.P\forkindep^l_{N.P} B.P$.
  We know that $T$ is simple, so by base monotonicity of Kim-independence in $P$, $P_A\forkindep^K_{P_M} P_B$ implies $P_N.P_A\forkindep^K_{P_N} P_B$.
\end{proof}

\begin{corollary}
    $\ACF_{\mathrm{PSF}}$ is simple, where PSF is the theory of pseudo-finite fields (see \cref{psf simple} for an alternative proof).
\end{corollary}

\subsection{Stability}
There are a few ways to prove that if $T$ is stable, then $\ACF_T$ is stable.
The first option, continuing in the path of the previous results, is using a Kim-Pillay style characterization on non-forking independence, which in simple theories is the same as Kim-independence over models.

The second option is a more direct approach, by counting types.
The second option will give us a stronger result, that if $T$ is $\lambda$-stable, then so is $\ACF_T$, which will let us extend to super-stability and $\omega$-stability.
Even though the second option is strictly stronger than the first, we will also show the first, to complete the picture on Kim-independence.

A third way to prove stability, is by proving the existence of saturated models of certain cardinalities.
This could be done using the characterization of saturated models of $\ACF_T$ found in \cref{saturated}, but we will not expand on it here.

\begin{remarkcnt} \label{strongly minimal stable}
    When the predicate has no extra structure, stability can also be deduced from \cite[Corollary 5.4]{Casanovas2001stable} (which cites \cite{Pillay1998lang}, probably meaning Proposition 3.1 there), which is a much more general statement: if $M$ is strongly minimal and $A$ is some subset of M such that the induced structure on $A$ is stable, then $(M,A)$ is stable.
\end{remarkcnt}

\begin{theorem} \label{stable}
    If $T$ is stable, then $\ACF_T$ is stable. 
\end{theorem}

\begin{proof}
    Suppose $T$ is stable, in particular $T$ is simple so \cref{simple} implies that $\ACF_T$ is simple.
    \cite[Proposition 8.4]{Kaplan2017} says that in simple theories, non-forking independence over models is the same as Kim-independence.
    To show that $\ACF_T$ is stable, it is enough to show that non-forking independence has stationarity over models (\cite[Theorem 12.22]{casanovas_2012}).
    Let $A$, $A'$ and $B$ be small subsets such that $M\subseteq A\cap A'\cap B$. 
    Suppose $A\forkindep^K_M B$, $A'\forkindep^K_M B$ and $A\equiv_M A'$.
    Without loss of generality we can assume $A$, $A'$ and $B$ are $\acl$-closed. Let $\alpha:A\to A'$ be an $L^P$-elementary map fixing $M$ pointwise. We want to extend $\alpha$ to an automorphism fixing $B$ pointwise. 

    By \cref{elementary map} $\alpha|_{P_A}$ is an $L$-elementary map in $P$, and by \cref{kim independence} $P_A\forkindep^K_{P_M} P_B$ and $P_{A'}\forkindep^K_{P_M} P_B$ in $P$.
    We know that $T$ is stable, so by stationarity $P_A\equiv_{P_B} P_{A'}$, i.e., $(a)_{a\in P_A} \equiv_{P_B} (\alpha(a))_{a\in P_A}$.
    Let $\sigma_0$ be an automorphism of $P$ mapping $P_A$ to $P_{A'}$ extending $\alpha|_{P_A}$ and preserving $P_B$ pointwise.
    We have $B\forkindep^l_{P_B} P$, so by stationarity of linear disjointedness we can extend $\sigma_0$ to $\sigma_1:B.P\to B.P$ preserving $B$ pointwise.
    By \cref{kim independence,linear disjointedness AP}, $A\forkindep^l_{M.P_A} B.P$ and $A'\forkindep^l_{M.P_{A'}} B.P$, so by \cref{isomorphism amalgamation} we can extend $\sigma_1$ and $\alpha$ to $\sigma_2:A.B.P\to A'.B.P$.
    Extend $\sigma_2$ to $\sigma_3$, an automorphism of $\M$, using \cref{extend to automorphism}.
    Since $\sigma_3$ extends $\alpha$ and fixes $B$ pointwise we are done. 
\end{proof}

\begin{corollary}
    $\ACF_\mathrm{SCF}$ is stable, where SCF is the theory of separably closed fields.
\end{corollary}

To prove stability by counting types, we will need to show that $P$ is stably embedded in $\M$.

\begin{definition}
  A set $Q\subseteq \M^m$ which is definable over the empty set is called \emph{stably embedded} if for every $n$, if $D\subseteq \M^{mn}$ is definable, then $D\cap Q^n$ is definable with parameters from $Q$.
\end{definition}

\begin{fact}[{\cite[Appendix, Lemma 1]{Chatzidakis1999}\label{condition stably embedded}}]
  For $Q\subseteq \M$ as above, if every automorphism of the induced structure on $Q$ lifts to an automorphism of $\M$, then $Q$ is stably embedded.   
\end{fact}

\begin{remarkcnt}
    The precise formulation of the above fact is more general but requires extra assumptions on $T$, namely that $T=T^{eq}$ and that the language is countable.
    However, those assumptions are not used in the proof of the direction we cited.
\end{remarkcnt}

\begin{lemma} \label{induced structure}
  The induced structure on $P$ as a subset of $\M$ is the same (up to interdefinability) as the intrinsic $L$-structure of $P$. 
\end{lemma}

\begin{proof}
    If $A\subseteq P^n$ is definable in $P$ by a formula $\phi\in L$, then we can construct by \cref{corresponding formula} a formula $\phi^P\in L^P$ that defines $A$ in $\M$.
    
    In the other direction, if $A\subseteq P^n$ is definable in $\M$ by a formula $\psi\in L^P$, then we can assume by \cref{bounded} that $\psi$ is bounded. 
    Remove any occurrence of $P$ in $\psi$, by replacing $x\in P$ with a tautology ($x=x$), to get a formula in $L$ that defines $A$ in $P$.
    
    This can also be deduced from \cref{extend to automorphism} using compactness (since \cref{extend to automorphism} implies that if $a, b \in P$ and $a \equiv b$ in $L$, then $a \equiv b$ in $L^P$ which implies the lemma using e.g. \cite[Lemma 3.1.1]{tent_ziegler_2012}).
 \end{proof}

From \cref{condition stably embedded,extend to automorphism,induced structure}\footnote{We only need the ``easy'' direction of \cref{induced structure}, i.e. that the $L$-structure is a reduct of the induced structure.} we conclude the following:

\begin{corollary}\label{stably embedded}
    $P$ is stably embedded in $\M$.  
\end{corollary} 

\begin{remarkcnt} \label{uniformly stably embedded}
    It follows from a simple compactness argument that $P$ is even \emph{uniformly} stably embedded, that is, for any formula $\phi(x,y)$ there exists a formula $\psi(x,z)$ such that for every $b\in \M$ there is $c\in P$ with $\phi(P,b)=\psi(P,c)$.
\end{remarkcnt}

\begin{theorem} \label{lambda stability}
    If $T$ is $\lambda$-stable, then $\ACF_T$ is $\lambda$-stable.
\end{theorem}

\begin{proof}
    Suppose $T$ is $\lambda$-stable, we can assume that $\abs{T}\le \lambda$ by replacing $T$ with an interdefinable theory (see e.g. \cite[Exercise 5.2.6]{tent_ziegler_2012}).
    Let $C\subseteq \M$ be a subset with $\abs{C}\le \lambda$, we need to prove that $\abs{S_1^{\ACF_T}(C)}\le \lambda$, where $S_1^{\ACF_T}(C)$ is the space of types in one variable over $C$.
    First we will prove that all elements in $\M\setminus\overline{P(C)}$ have the same type over $C$ in $\ACF_T$.
    Suppose $a_0,a_1\in\M\setminus\overline{P(C)}$, that is both $a_0$ and $a_1$ are transcendental over $P(C)$.
    There is an isomorphism of fields $P(C,a_0)\to P(C,a_1)$ given by fixing $P(C)$ pointwise and mapping $a_0\mapsto a_1$.
    By \cref{extend to automorphism}, we can extend this map to an automorphism of $\M$, so $a_0\equiv_C a_1$ in $\ACF_T$.
    
    It remains to show that there are at most $\lambda$ types in $\overline{P(C)}$.
    Any element of $\overline{P(C)}$ solves some non-zero polynomial of the form $q(x;b,c)$ with $b\in P^n$ and $c\in C^m$, and in particular satisfies
    $$\phi(x;c)=\exists y\in P\ (q(x;y,c)=0\land \exists x' q(x';y,c)\ne 0).$$
    Thus, any type in $\overline{P(C)}$ contains some formula $\phi(x;c)$ as above.
    There are at most $\lambda$ formulas in $L^P$ with parameters from $C$, so it is enough to prove that there are at most $\lambda$ types that contain any given formula $\phi(x;c)$ as above.
    
    First of all, $P$ is stably embedded in $\M$ (\cref{stably embedded}), so every $C$-definable subset of $P^n$ in $\ACF_T$ is also definable in $\ACF_T$ with parameters from $P$. 
    Let $D\subseteq P$ be the set of all the parameters needed to define every $C$-definable subset of $P^n$.
    There are at most $\lambda$ definable subsets of $P^n$ over $C$, so $\abs{D}\le \lambda$.
    
    Let $[\phi]\subseteq S_1^{\ACF_T}(C)$ be the set of types implying $\phi(x;c)$. We will construct a map $\rho:[\phi]\to S^T_n(D)$ such that $\rho$ has finite fibers.
    Because $T$ is $\lambda$-stable, $\abs{S^T_n(D)}\le \lambda$, so this will imply $\abs{[\phi]}\le \lambda$ as needed.
    
    For any type $p(x)\in [\phi]$, choose some realization $a\vDash p$. 
    In particular, $\vDash \phi(a;c)$, so we can choose some $b\in P^n$ such that $q(x;b,c)$ is non-zero and $q(a;b,c)=0$.
    Define $\rho(p)=\tp^T(b/D)$.
    Suppose $p_0,p_1\in[\phi]$ and $\rho(p_0)=\rho(p_1)$, that is, if $a_i,b_i$ are the specific elements we chose for $p_i$ ($i=0,1$), then $b_0\equiv_D b_1$ in $T$.
    There is an automorphism of $P$ over $D$ mapping $b_0\mapsto b_1$, which can be extended by \cref{extend to automorphism} to an automorphism of $\M$ over $D$, so $b_0\equiv_D b_1$ in $\ACF_T$. 
    We want to prove that $b_0\equiv_C b_1$ in $\ACF_T$.
    Suppose $b_0$ belongs to some $C$-definable set, we can assume that it is a subset of $P^n$ because $b_0\in P^n$.
    By the construction of $D$, this $C$-definable subset of $P^n$ is also $D$-definable in $\ACF_T$, so $b_1$ belongs to it as $b_0\equiv_D b_1$ in $\ACF_T$.
    
    Let $\sigma\in \mathrm{Aut}(\M/C)$ be an automorphism mapping $b_0$ to $b_1$.
    In particular $q(\sigma(a_0);b_1,c)=0$, thus $a_0$ has the same type over $C$ as a root of $q(x;b_1,c)$, specifically $\sigma(a_0)$.
    It follows that every type in the fiber of $\rho(p_1)$ is a type over $C$ of a root of $q(x;b_1,c)$, however $q(x;b_1,c)$ is non-zero, so it has only finitely many roots. 
    Thus, $\rho$ has finite fibers.
\end{proof}

 We can apply \cref{lambda stability} to specific $\lambda$'s to give another proof of \cref{stable}.
 We also get the following corollaries:
 
\begin{corollary}\label{superstable}
    If $T$ is superstable, then $\ACF_T$ is superstable.
\end{corollary}

\begin{corollary} \label{omega-stable}
    If $T$ is $\omega$-stable, then $\ACF_T$ is $\omega$-stable.
\end{corollary} 

\begin{corollary}
    $\ACF_\ACF$ is $\omega$-stable, see \cref{tuples of ACF} for an extended application of this result.
\end{corollary}

\begin{remarkcnt}
    By \cite{Poizat1983}, $\ACF_\ACF$ is a \emph{belle pair} (see there for the definition), and it is stable.
    In \cite{BENYAACOV2003235}, the notion of belle pairs was expanded to \emph{lovely pairs} and a description of non-forking independence was given.
    When considering pairs of ACF, the description of non-forking independence in \cref{kim independence} is slightly different from the description given in \cite[Proposition 7.3]{BENYAACOV2003235} --- instead of the condition $A.P\forkindep^l_{M.P} B.P$ they have $A.P \forkindep^\ACF_{M.P} B.P$.
    However, in this case the conditions are equivalent, as can be seen in \cite[Corollary 6.2]{MARTIN_PIZARRO_2020}.
\end{remarkcnt}

\subsection{NIP}
We will prove that if $T$ is NIP, then $\ACF_T$ is NIP. 
First we will define the notions of a NIP formula, type and theory, and present some basic facts based on \cite{simon_2015} and \cite{kaplan2014dp-rank}.

\begin{definition} \label{definition NIP}
  Suppose that $T$ is some theory.
  A formula $\phi(x,y)$ has the \emph{independence property (IP)} if there is a sequence $(a_i)_{i<\omega}$ (in a model of $T$) such that for every $s\subseteq \omega$ the set $\set{\phi(a_i,y) \mid i\in s}\cup\set{\lnot\phi(a_i,y)\mid i\notin s}$ is consistent.
  
  A partial type $\pi(x)$ has \emph{IP} if there is a formula $\phi(x,y)$ and a sequence $(a_i)_{i<\omega}$ of realizations $a_i\vDash \pi(x)$ such that for every $s\subseteq \omega$ the set $\set{\phi(a_i,y)\mid i\in s}\cup\set{\lnot\phi(a_i,y)\mid i\notin s}$ is consistent.
  Otherwise, $\pi(x)$ is \emph{NIP}.
  
  The theory $T$ has \emph{IP} if some formula has IP, or equivalently the type $x=x$ has IP. 
  Otherwise, $T$ is \emph{NIP}.
\end{definition}

\begin{fact}[{\cite[Lemma 2.7]{simon_2015}}]\label{NIP even}
    A formula $\phi(x,y)$ has IP iff there is an indiscernible sequence $(a_i)_{i<\omega}$ and a tuple $b$ such that $\models\phi(a_i,b)\iff i\text{ is even}$.
\end{fact}
 
\begin{fact} [{\cite[Proposition 2.11]{simon_2015}}]\label{NIP singleton}
    A theory $T$ is NIP iff no formula $\phi(x,y)$ with $\abs{y}=1$ has IP.
\end{fact} 

\begin{fact}[{\cite[Proposition 2.6]{kaplan2014dp-rank}}]\label{NIP end segment}
  Suppose $\pi(x)$ is a partial NIP type over $A$ and $B$ is a set of realizations of $\pi(x)$.
  If $I=(a_i)_{i<\abs{T}^+ +\abs{B}^+}$ is an $A$-indiscernible sequence, then some end segment of $I$ is indiscernible over $AB$.
\end{fact}
 
First we need to show that $P$ is NIP as in \cref{definition NIP}
 
\begin{lemma} \label{P NIP}
   If $T$ is NIP, then  $P$ is NIP, i.e. the partial type $x\in P$ is NIP.
\end{lemma}

\begin{proof}
    Suppose $x\in P$ has IP.
    Then there are a sequence $(a_i)_{i<\omega}$ with $a_i\in P$ and a formula $\phi(x,y)$, such that for every $s\subseteq \omega$, there exists $b_s\in \M$ such that $\M\vDash \phi(a_i,b_s)\iff i\in s$.
    By \cref{uniformly stably embedded}, $P$ is uniformly stably embedded in $\M$, so there exists a formula $\psi(x,z)\in L^P$ and parameters $c_s\in P$ for every $s\subseteq \omega$, such that $\phi(P,b_s)=\psi(P,c_s)$, and in particular $\M\vDash \psi(a_i,c_s)\iff i\in s$.
    
    The induced structure on $P$ is interdefinable with the internal $L$-structure of $P$ (\cref{induced structure}), so there is some formula $\psi'(x,z)\in L$ that defines the same set in $P$ as $\psi(x,z)$, in particular $P\vDash \psi'(a_i,c_s)\iff i\in s$.
    The formula $\psi'(x,y)$ has IP in $P\vDash T$, in contradiction to $T$ being NIP.
\end{proof}

\begin{theorem}\label{NIP}
    If $T$ is NIP, then $\ACF_T$ is NIP.
\end{theorem}

\begin{proof}
    Suppose $\ACF_T$ has IP, by \cref{NIP singleton} there is some $\phi(x,y)$ with $\abs{y}=1$ that has IP.
    Using \cref{NIP even} and compactness, there is an indiscernible sequence $I=(a_i)_{i<\abs{T}^+}\subseteq \M$ and some $c\in \M$ such that $\M\vDash \phi(a_i,c)\iff i\text{ is even}$.
    
    First consider the case where $c$ is transcendental over $P(I)$.
    In particular, $c$ is transcendental over $P(a_0)$ and $P(a_1)$.
    There is an automorphism mapping $a_0$ to $a_1$, as they have the same type over the empty set.
    Apply this automorphism on $c$ to get $c'$ which is transcendental over $a_1$.
    Both $c$ and $c'$ are transcendental over $P(a_1)$, so by \cref{extend to automorphism} $c$ and $c'$ have the same type over $P(a_1)$ in ACF$_T$. 
    This is a contradiction, as we have $\vDash \phi(a_1, c')$ and $\vDash \lnot\phi(a_1, c)$.
    
    Now consider the case where $c$ is algebraic over $P(I)$.
    There is some finite subsequence $I_0\subseteq I$ and some finite tuple $b\in P$, such that $c$ is algebraic over $I_0b$.
    Let $I'\subseteq I$ be some end segment starting after $I_0$; note that $I'$ is indiscernible over $I_0$.
    As $P$ is NIP (\cref{P NIP}), by \cref{NIP end segment} there is an end segment $I''\subseteq I'$ that is indiscernible over $I_0b$.
    It follows that $I''$ is also indiscernible over $\acl(I_0b)$, and in particular over $c$, a contradiction.
\end{proof}
 
\begin{corollary} \label{acvf nip}
    Let ACVF be the theory of algebraically closed valued fields in the divisibility language, that is the language of rings with a binary relation $x|y$ signifying $v(x)<v(y)$.
    ACVF is NIP, so $\ACF_\mathrm{ACVF}$ is NIP.
\end{corollary}
 
\begin{remarkcnt}
    One could also use a counting type approach to prove preservation of NIP, similar to the proof of \cref{lambda stability}.
    This would require working in a generic extension of ZFC such that $\mathrm{ded}(\kappa)^{\aleph_0}<2^\kappa$ for some 
    infinite cardinal $\kappa$ (where $\mathrm{ded}(\kappa)$ is the supremum of cardinalities of linear orders with a dense subset of size $\le \kappa$). 
    For an expanded explanation of this approach, see \cite[Theorem II.4.10]{shelah1990classification} and \cite[Corollary 24]{Adler07introductionto}.
    
    Alternatively, one could also apply more general results, i.e., \cite[Corollary 2.5]{chernikov2011externally} and \cite[Proposition 2.5]{jahnke2019nip}, but we chose to give a direct argument.
\end{remarkcnt}

\section{Applications}\label{subfield applications}

In this section we will apply the above results to specific theories.

\subsection{Tuples of algebraically closed fields}

In this section we will consider (perhaps infinite) chains of algebraically closed fields, which, for the finite case, is a particular case of \emph{beaux uples} in the sense of \cite{bouscaren1988beaux}. 
The main result of this section is \cref{tuples of acf order type} which classifies the theories of such chains based on the order type of the chain.

\begin{definition}
   For any ordered set $I$, define $L^I=L_{\mathrm{rings}}\cup\set{P_i}_{i\in I}$ with $P_i$ unitary predicates and define the theory $\ACF^I$ expanding $\ACF$ in $L^I$, such that: 
   \begin{enumerate}
       \item Each $P_i$ is an algebraically closed field, that is strictly contained in the model.
       \item For $i<j$, $P_i\subsetneq P_j$.
   \end{enumerate}
   In particular, $\ACF^n$ is the theory of algebraically closed fields $M$, with $n$ algebraically closed subfields $P_0\subsetneq P_1\subsetneq\dots\subsetneq P_{n-1}\subsetneq M$.
\end{definition}

\begin{proposition} \label{tuples of ACF}
    Let $I$ be any ordered set.
    \begin{enumerate}
        \item The completions of $\ACF^I$ are given by fixing the characteristic, $\ACF_p^I$.
        \item Every completion of $\ACF^I$ is stable.
    \end{enumerate}
\end{proposition}

\begin{proof}
    We will first prove for $I=n$, by induction on $n$.
    For $n=0$, $\ACF^0=\ACF$, and indeed the completions of $\ACF$ are given by fixing the characteristic and every completion $\ACF_p$ is stable.
    Suppose it is true for $n$.
    We have $\ACF^{n+1}=\ACF_{\ACF^n}$, where we denote the added predicate by $P_n$.
    By \cref{complete}, the completions of $\ACF^{n+1}$ are given by completions of $\ACF^n$, which are given by fixing the characteristic.
    Furthermore, $\ACF^{n+1}_p=\ACF_{\ACF^n_p}$, so
    by \cref{stable} every completion $\ACF^{n+1}_p$ is stable.
    
    Now consider a general ordered set $I$ and fix a characteristic $\ACF^I_p$.
    Let $\phi$ be a sentence in $L^I$ and let $I_\phi\subset I$ be the subset of indexes $i\in I$ such that $P_i$ appears in $\phi$.
    $I_\phi$ is finite, suppose $I_\phi=\set{i_0<\dots<i_{n-1}}$.
    $\ACF^n_p$ is complete, so by renaming the predicates $P_0,\dots,P_{n-1}$ to $P_{i_0},\dots,P_{i_{n-1}}$ we get that $\ACF_p^{I_\phi}$ is complete.
    Thus, $\ACF_p^{I_\phi}\vdash \phi$ or $\ACF_p^{I_\phi}\vdash \lnot\phi$, but $\ACF_p^{I_\phi}$ is a restriction of $\ACF_p^I$, so $\ACF_p^I\vdash \phi$ or $\ACF_p^I\vdash \lnot\phi$.
    The completions $\ACF^I_p$ are all the completions of $\ACF^I$, because any completion has to fix a characteristic so it must extend some $\ACF^I_p$.
    
   We need to show that every completion $\ACF^I_p$ is stable.
    If $\phi\in L^I$ was a formula witnessing instability in $\ACF^I_p$, then it would witness instability in $\ACF^{I_\phi}_p$, which would imply that $\ACF^n_p$ is unstable for $n=\abs{I_\phi}$.
\end{proof}

We will further classify the stability of $\ACF^I_p$ (when is it $\omega$-stable, superstable or totally transcendental) based on the order type of $I$.
In the case that $I$ is an ordinal, we will need the following lemma.

\begin{lemma} \label{extend automorphism tuples of acf}
  Let $\alpha$ be an ordinal and $M\vDash \ACF^\alpha$.
  Any $L^\beta$-automorphism of $P_\beta$ for $\beta<\alpha$ can be extended to an $L^\alpha$-automorphism of $\M$.
\end{lemma}

\begin{proof}
    Let $\sigma_\beta$ be an automorphism of $P_\beta$, we will construct by transfinite induction on $\beta\le\gamma<\alpha$ automorphisms $\sigma_\gamma$ of $P_\gamma$, such that if $\beta\le \gamma'<\gamma<\alpha$, then $\sigma_{\gamma}$ extends $\sigma_{\gamma'}$.
    
    Let $\beta\le \gamma<\alpha$ and suppose we constructed $\sigma_{\gamma'}$ for $\beta\le \gamma'<\gamma$.
    Let $\sigma_{<\gamma}$ be the union of $\set{\sigma_{\gamma'}}_{\beta\le \gamma'<\gamma}$, $\sigma_{<\gamma}$ is a field automorphism of $P_{<\gamma}=\bigcup_{\gamma'<\gamma} P_{\gamma'}$ (if $\gamma=\gamma'+1$ is a successor ordinal, then $\sigma_{<\gamma}=\sigma_{\gamma'}$).
    Let $S$ be a transcendence basis of $P_\gamma$ over $P_{<\gamma}$, extend $\sigma_{<\gamma}$ to a field automorphism $\sigma_\gamma$ by fixing $S$ pointwise and extending to the algebraic closure. 
    For every $\gamma'<\gamma$, $\sigma_\gamma$ preserves $P_{\gamma'}$ setwise, so $\sigma_\gamma$ is an $L^\gamma$-automorphism.
    
    Once we constructed $\sigma_\gamma$ for every $\beta\le \gamma<\alpha$, we can construct $\sigma_\alpha$, an $L^\alpha$-automorphism of $M$, in a similar fashion: take $\sigma_{<\alpha}$ the union of $\set{\sigma_\gamma}_{\beta\le \gamma<\alpha}$, fix a transcendence basis of $M$ over $P_{<\alpha}$ pointwise and extend to the algebraic closure.
\end{proof}

\begin{proposition} \label{tuples of acf order type}
    For an ordered set $I$:
    \begin{enumerate}
        \item If $I$ is finite, or countable and well-ordered, then every completion of $\ACF^I$ is $\omega$-stable.
        \item If $I$ is uncountable and well-ordered, then every completion of $\ACF^I$ is totally transcendental, and in particular superstable, but not $\omega$-stable.
        \item If $I$ is not well-ordered, then no completion of $\ACF^I$ is superstable.
    \end{enumerate}
\end{proposition}
    
\begin{proof}
    Fix a completion $\ACF^I_p$ (by \cref{tuples of ACF}).
    
    (1) The theory $\ACF^I_p$ depends only on the order type of $I$, up to renaming predicates, so it is enough to prove for $I=\alpha$ a finite or countable ordinal.
    We will prove that $\ACF^\alpha_p$ is $\omega$-stable by transfinite induction on $\alpha<\omega_1$.
    For $\alpha=0$, $\ACF^0_p=\ACF_p$ is $\omega$-stable.
    If $\ACF_p^\alpha$ is $\omega$-stable, then note that $\ACF_p^{\alpha+1}=\ACF_{\ACF_p^\alpha}$ where we name the added predicate $P_\alpha$, so by \cref{omega-stable} $\ACF_p^{\alpha+1}$ is $\omega$-stable.
    
    Suppose that $\alpha$ is a countable limit ordinal and for every $\beta<\alpha$, $\ACF_p^\beta$ is $\omega$-stable, the proof that $\ACF_p^\alpha$ is $\omega$-stable will be similar to the proof of \cref{lambda stability}.
    Let $\M\vDash \ACF_p^\alpha$ be a monster model and let $C\subseteq \M$ be a countable subset.
    Denote $P_{<\alpha}=\bigcup_{\beta<\alpha} P_\beta$. 
    First we will show that every two elements in $\M\setminus \overline{P_{<\alpha}(C)}$ have the same type over $C$.
    Let $a_0,a_1\in\M\setminus \overline{P_{<\alpha}(C)}$, for every  $\beta<\alpha$, $a_0$ and $a_1$ are transcendental over $P_\beta(C)$ so by \cref{extend to automorphism} there is an automorphism of $\M\restriction{L^{\beta+1}}$ preserving $P_\beta(C)$ and mapping $a_0\mapsto a_1$.
    Thus, $a_0\equiv_C a_1$ in $L^{\beta+1}$ for every $\beta<\alpha$, so $a_0\equiv_C a_1$ in $L^\alpha$, as every formula in $L^\alpha$ belongs to some $L^{\beta+1}$ where $\beta$ is the largest ordinal such that $P_\beta$ appears in the formula.
    
    Now we will show that there at most countably many types over $C$ realized in $\overline{P_{<\alpha}(C)}$.
    Any element $a\in \overline{P_{<\alpha}(C)}$ solves some non-zero polynomial of the form $q(x;b,c)$ with $b\in P_{<\alpha}^n$ and $c\in C^m$. 
    There is some $\beta<\alpha$ such that $b\in P_\beta^n$, in particular $a$ satisfies
    $$\phi(x;c)=\exists y\in P_\beta\ (q(x;y,c)=0\land \exists x' q(x';y,c)\ne 0).$$
    Thus, any type in $\overline{P_{<\alpha}(C)}$ contains some formula $\phi(x;c)$ as above.
    There are countably many formulas in $L^\alpha$ with parameters from $C$, so it is enough to prove that there are at most countably many types that contain any given formula $\phi(x;c)$ as above.
    
    First of all, $P_\beta$ is stably embedded in $\M$ (every automorphism of $P_\beta$ can be extended to an automorphism of $\M$ so we can use \cref{condition stably embedded}; alternatively, $\ACF_p^\alpha$ is stable so every definable subset is stably embedded), so every $C$-definable subset of $P_\beta^n$ is also definable in $\ACF_p^\alpha$ with parameters from $P_\beta$. 
    Let $D\subseteq P_\beta$ be the set of all the parameters needed to define every $C$-definable subset of $P_\beta^n$.
    There are at most countably many definable subsets of $P_\beta^n$ over $C$, so $D$ is countable.
    
    Let $[\phi]\subseteq S_1^{\ACF_p^\alpha}(C)$ be the set of types implying $\phi(x;c)$ as above, we will construct a map $\rho:[\phi]\to S^{\ACF_p^\beta}_n(D)$ such that $\rho$ has finite fibers.
    Because $\ACF_p^\beta$ is $\omega$-stable, $\abs{S^{\ACF_p^\beta}_n(D)}$ is countable, so this will imply that $[\phi]$ is countable as needed.
    
    For any type $p(x)\in [\phi]$, choose some realization $a\vDash p$. 
    In particular, $\vDash \phi(a;c)$, so we can choose some $b\in P_\beta^n$ such that $q(x;b,c)$ is non-zero and $q(a;b,c)=0$.
    Define $\rho(p)=\tp^{\ACF_p^\beta}(b/D)$.
    Suppose $p_0,p_1\in[\phi]$ and $\rho(p_0)=\rho(p_1)$, that is, if $a_i$ and $b_i$ are the specific elements we chose for $p_i$ ($i=0,1$), then $b_0\equiv_D b_1$ in $\ACF_p^\beta$.
    There is an automorphism of $P_\beta$ over $D$ mapping $b_0\mapsto b_1$, which can be extended by \cref{extend automorphism tuples of acf} to an automorphism of $\M$ over $D$, so $b_0\equiv_D b_1$ in $\ACF_p^\alpha$. 
    We want to prove that $b_0\equiv_C b_1$ in $\ACF_p^\alpha$.
    Suppose $b_0$ belongs to some $C$-definable set, we can assume that it is a subset of $P_\beta^n$ because $b_0\in P_\beta^n$.
    By the construction of $D$, this $C$-definable subset of $P_\beta^n$ is also $D$-definable in $\ACF_p^\alpha$, so $b_1$ belongs to it as $b_0\equiv_D b_1$ in $\ACF_p^\alpha$.
    
    Let $\sigma\in \mathrm{Aut}(\M/C)$ be an automorphism mapping $b_0\mapsto b_1$.
    In particular $q(\sigma(a_0);b_1,c)=0$, thus $a_0$ has the same type over $C$ as a root of $q(x;b_1,c)$, specifically $\sigma(a_0)$.
    It follows that every type in the fiber of $\rho(p_1)$ is a type over $C$ of a root of $q(x;b_1,c)$, however $q(x;b_1,c)$ is non-zero, so it has only finitely many roots. 
    Thus, $\rho$ has finite fibers.
    
    (2) Suppose $I$ is uncountable and well-ordered.
    If $\ACF_p^I$ was not totally transcendental, there would be a binary tree of consistent formulas $\set{\phi_s(x;c_s)}_{s\in 2^{<\omega}}$ (see \cite[Definition 5.2.5]{tent_ziegler_2012}).
    Let $I_0\subseteq I$ be the finite or countable subset of indexes $i\in I$ such that $P_i$ appears in some formula $\phi_s$.
    The tree $\set{\phi_s(x;c_s)}_{s\in 2^{<\omega}}$ is also a binary tree of consistent formulas in $\ACF_p^{I_0}$, so $\ACF_p^{I_0}$ is not totally transcendental.
    However, a subset of a well-ordered set is also well-ordered, so by the previous part $\ACF_p^{I_0}$ is $\omega$-stable and in particular totally transcendental.
    
    However, $\ACF^I$ can not be $\omega$-stable, as it is not interdefinable with a theory in a countable language --- each $P_i$ for $i\in I$ is a distinct definable set.
    
    (3) Note that an ordered set $I$ is well-ordered iff $I$ does not contain an infinite descending chain.
    If $I$ is not well-ordered, let $(i_k)_{k<\omega}\subseteq I$ be a descending chain, then $(P_{i_k})_{k<\omega}$ is a descending chain of definable subfields in $\ACF_p^I$.
    Considering only the additive group structure, $(P_{i_k})_{k<\omega}$ is a descending chain of definable subgroups each of infinite index in the previous one, so $\ACF_p^I$ is not superstable (see e.g. \cite[Exercise 8.6.10]{tent_ziegler_2012}).
\end{proof}

\subsection{Complete system of a Galois group}
For a profinite group $G$ one can associate a structure $S(G)$, called  the complete system of $G$, in a multi-sorted language.
This definition is due to \cite{CDM81}, we will present the definition as given in \cite[Definition 7.1.6]{ramsey2018}.

\begin{definition}
  Suppose $G$ is a profinite group.
  Let $\mathcal{N}(G)$ be the collection of open normal subgroups of $G$.
  Define
  $$S(G)=\coprod_{N\in \mathcal{N}(G)} G/N.$$
  Let $L_G$ be the language with a sort $X_n$ for each $n<\omega$, two binary relation symbols $\le$, $C$ and a ternary relation $P$.
  We regard $S(G)$ as an $L_G$-structure in the following way:
  \begin{itemize}
      \item The coset $gN$ is in the sort $X_n$ iff $[G:N]\le n$.
      \item $gN\le hM$ iff $N\subseteq M$.
      \item $C(gN,hM)$ iff $N\subseteq M$ and $gM=hM$.
      \item $P(g_1N_1,g_2N_2,g_3N_3)$ iff $N_1=N_2=N_3$ and $g_1g_2N_1=g_3N_1$.
  \end{itemize}
  Note that we do not require the sorts to be disjoint (see \cite[\S 1]{Chatzidakis1998} for a discussion on the syntax of this structure).
\end{definition}

For a field $F$, let $G(F)=\mathrm{Gal}(\overline{F}/F)$ be the absolute Galois group of $F$, which is profinite.
In \cite[Corollary 7.2.7]{ramsey2018}, Ramsey proved that if $F$ is a $\PAC$ field such that $\Th(S(G(F))$ is $\NSOP_1$, then $\Th(F)$ is $\NSOP_1$.
We will prove the other direction, using the following fact, proved in \cite[Proposition 5.5]{chatzidakis2002}.

\begin{fact} \label{complete system interpretable}
    $S(G(F))$ is interpretable in $(K,F)$ where $K$ is any algebraically closed field extending $F$.
\end{fact}

\begin{proposition}\label{galois group nsop1}
  Let $F$ be a PAC field. Then $\Th(F)$ is $\NSOP_1$ iff $\Th(S(G(F)))$ is $\NSOP_1$. 
\end{proposition}

\begin{proof}
    The left to right direction is \cite[Corollary 7.2.7]{ramsey2018}

    For the right to left direction, let $K\supseteq F$ be a large enough algebraically closed extension, $(K,F)\vDash \ACF_{\Th(F)}$.
    From \cref{nsop1} $\ACF_{\Th(F)}$ is $\NSOP_1$, but from \cref{complete system interpretable} $S(G(F))$ is interpretable in $(K,F)$, so $\Th(S(G(F))$ is $\NSOP_1$.
\end{proof}

\subsection{Pseudo finite fields}
Pseudo finite fields were first studied in \cite{Ax1968}, we will give the definition from \cite{tent_ziegler_2012}.

\begin{definition} \label{def PSF}
  Suppose $F$ is a field.
  We say that $F$ is \emph{pseudo-algebraically closed} if every absolutely irreducible variety over $F$ has an $F$-rational point, or equivalently if it is existentially closed in every regular extension.
  We say that $F$ is \emph{pseudo-finite} if it is perfect, pseudo-algebraically closed and 1-free (has exactly one extension of degree $n$ for every $n$).
  Being pseudo-algebraically closed or pseudo-finite is an elementary property \cite[Corollary B.4.3, Remark B.4.12]{tent_ziegler_2012}, so there are first-order theories $\PAC$, $\PSF$ of pseudo-algebraically closed, pseudo-finite fields respectively.
\end{definition}

\begin{proposition} \label{psf model complete}
  $\ACF_\PSF^{ld}$ is model complete.
\end{proposition}

\begin{proof}
    If $Q$ and $R$ are pseudo-finite fields such that $Q\subseteq R$ is a relatively algebraically closed extension, that is $\overline{Q}\cap R=Q$, then $Q\preceq R$ \cite[Proposition 20.10.2]{fried2008field}.
    In particular, if $Q\subseteq R$ is a regular extension, then it is relatively algebraically closed, so $Q\preceq R$.
    Thus, by \cref{model completeness}, $\ACF_\PSF^{ld}$ is model complete.
\end{proof}

\begin{proposition}\label{psf simple}
  Every completion of $\ACF_\PSF$ is simple.
\end{proposition}

\begin{proof}
    By \cref{complete}, completions of $\ACF_\PSF$ are given by completions of $\PSF$, which are simple by \cite[Corollary 7.5.6]{tent_ziegler_2012}, so the result follows from \cref{simple}.
    We will give another more direct proof using $\ACFA$, the model companion of difference fields, which is simple \cite[Example 2.6.9]{kim2014simplicity}.
    
    Let $(M,P)\vDash \ACF_\PSF$.
    We will show that there is an automorphism $\sigma\in \mathrm{Gal}(\overline{P}/P)$ such that $\mathrm{Fix}(\sigma):=\set{a\in \overline{P}\mid \sigma(a)=a}=P$. 
    Consider $P_n$ the unique cyclic extension of degree $n$ of $P$ and $\sigma_n$ a generator of $\mathrm{Gal}(P_n/P)$. 
    The fixed field of $\sigma_n$ is $P$, so the inverse limit of $\sigma_n$ is an automorphism of $\bar{P}$ whose fixed field is $P$.
    
    By \cite[Corollary 1.2]{Afshordel2014GENERICAW}, we can embed $(\overline{P},\sigma)$ into $(N,\sigma')$ a model of $\ACFA$, with $\mathrm{Fix}(\sigma')=P$.
    The structure $(N,P)$ is a reduct of $(N,\sigma')$, so it is simple.
    The structures $(M,P)$, $(N,P)$ and $(\overline{P},P)$ are models of $\ACF_\PSF$, and they can be uniquely expanded to models of $\ACF_\PSF^{ld}$. 
    \cref{submodel} implies that $(\overline{P},P)\subseteq (M,P)$, $(\overline{P},P)\subseteq (N,P)$ are substructures in $\ACF_\PSF^{ld}$, because they all share the same predicate.
    However, \cref{psf model complete} says that $\ACF_\PSF^{ld}$ is model complete, so those are elementary substructures.
    In particular, they are elementary substructures in $\ACF_\PSF$.
    Because $(N,P)$ is simple and $(\bar{P},P)\preceq (N,P)$, we get that $(\bar{P},P)$ is simple.
    But also $(\bar{P},P)\preceq (M,P)$, so $(M,P)$ is simple.
\end{proof}

\section{Questions}

There are several questions that arose in our work, which we did not address in this paper.

\begin{question}
    What other classification properties can we lift from $T$ to $\ACF_T$? $\mathrm{NTP}_2$, $\mathrm{NSOP}_n$ (for $n\ge 2$)?
\end{question}

\begin{question}
    What results still hold when we replace $\ACF$ in $\ACF_T$ with a different theory of fields?
    $\mathrm{SCF}$, $\mathrm{ACVF}$?
    The theory of dense pairs of $\mathrm{ACVF}$ was studied in \cite{Delon2012}. 
\end{question}

\begin{question}\label{question strongly minimal}
    What results still hold when we replace $\ACF$ in $\ACF_T$ with any strongly minimal theory?
    See \cref{strongly minimal stable}.
\end{question}

\bibliographystyle{alpha}
\bibliography{library}

\end{document}